\documentclass[a4paper,11pt]{article}

\usepackage[utf8]{inputenc}
\usepackage{amsfonts}
\usepackage{amssymb}
\usepackage{amsthm}
\usepackage{amsmath}
\usepackage{a4wide}
\usepackage{mathrsfs}
\usepackage{epsfig}
\usepackage{palatino}
\usepackage{tikz}
\usepackage{fp,ifthen}
\usepackage{nicefrac}
\usetikzlibrary{decorations.pathreplacing}
\usepackage{float}
\usepackage{enumerate} 
\usepackage{enumitem} 
\usepackage[a4paper,twoside,top=2.5cm, bottom=2cm, left=1.7cm, right=1.7cm]{geometry}

\usepackage[colorinlistoftodos,textwidth=2.3cm]{todonotes} 

\newcommand{\pdfgraphics}{\ifpdf\DeclareGraphicsExtensions{.pdf,.jpg}\else\fi}
\usepackage{graphicx}

\usepackage{color}
\definecolor{hanblue}{rgb}{0.27, 0.42, 0.81}
\definecolor{red}{rgb}{1.0, 0.0, 0.0}
\usepackage[colorlinks, citecolor=blue,linkcolor=blue, urlcolor = blue]{hyperref}
\usepackage[english,capitalize]{cleveref}

\definecolor{amber}{rgb}{1.0, 0.75, 0.0}
\definecolor{crimson}{rgb}{0.86, 0.08, 0.24}
\definecolor{cornflowerblue}{rgb}{0.39, 0.58, 0.93}
\definecolor{green(munsell)}{rgb}{0.0, 0.66, 0.47}

\usepackage{mathtools}

\theoremstyle{plain}

\newtheorem{teo}{Theorem}[section]
\newtheorem{lemma}[teo]{Lemma}
\newtheorem{prop}[teo]{Proposition}
\newtheorem{cor}[teo]{Corollary}

\theoremstyle{definition}
\newtheorem{defn}[teo]{Definition}

\theoremstyle{remark}
\newtheorem{rem}[teo]{Remark}

\numberwithin{equation}{section}

\newcommand{\de}{\ensuremath{\,\mathrm d}} 
\renewcommand{\d}{\ensuremath{\mathrm d}} 
\newcommand{\st}{\ensuremath{\ :\ }} 
\newcommand{\eqdef}{\ensuremath{\vcentcolon=}}
 
\renewcommand{\epsilon}{\varepsilon}
\newcommand{\N}{\ensuremath{\mathbb N}}
\newcommand{\Z}{\ensuremath{\mathbb Z}}
\newcommand{\R}{\ensuremath{\mathbb R}}

\DeclarePairedDelimiter\scal{\langle}{\rangle} 

\renewcommand{\div}{{\rm div}\,}
\newcommand{\Ha}{\mathcal{H}}
\newcommand{\tr}{{\rm tr}} 

\newcommand{\res}{\mathbin{\vrule height 1.6ex depth 0pt width
0.13ex\vrule height 0.13ex depth 0pt width 1.3ex}} 

\makeatletter
\renewcommand*\env@matrix[1][*\c@MaxMatrixCols c]{%
  \hskip -\arraycolsep
  \let\@ifnextchar\new@ifnextchar
  \array{#1}}
\makeatother

\begin{document}

\pdfgraphics 

\title{Minimizing properties of networks via global and local calibrations}

\author{Alessandra Pluda \footnote{\href{mailto:alessandra.pluda@unipi.it}{alessandra.pluda@unipi.it}, Dipartimento di Matematica, Universit\`a di Pisa, Largo Bruno Pontecorvo 5, 56127 Pisa, Italy.} \and Marco Pozzetta \footnote{\href{mailtto:marco.pozzetta@unina.it}{marco.pozzetta@unina.it}, Dipartimento di Matematica e Applicazioni, Universit\`a di Napoli Federico II, Via Cintia, Monte S. Angelo 80126 Napoli, Italy}}

\date{\today}

\maketitle

\vspace{-0.5cm}
\begin{abstract}
\noindent In this note we prove that minimal networks enjoy minimizing properties for the length functional. A minimal network is, roughly speaking, a subset of $\R^2$ composed of straight segments joining at triple junctions forming angles equal to $\tfrac23 \pi$; in particular such objects are just critical points of the length functional a priori. We show that a minimal network $\Gamma_*$: \emph{i)} minimizes mass among currents with coefficients in an explicit group (independent of $\Gamma_*$) having the same boundary of $\Gamma_*$, \emph{ii)} identifies the interfaces of a partition of a neighborhood of $\Gamma_*$ solving the minimal partition problem among partitions with same boundary traces. Consequences and sharpness of such results are discussed. The proofs reduce to rather simple and direct arguments based on the exhibition of (global or local) calibrations associated to the minimal network.
\end{abstract}

\textbf{Mathematics Subject Classification (2020)}: 
49Q20, 53C38, 49Q15, 49Q05.


\section{Introduction}

The Steiner problem
in its classical formulation (see for instance~\cite{CoRo}) reads as follows:
given $\mathcal{C}$ a collection of $n$ points $p_1,\ldots,p_n$ in the Euclidean plane, one wants
to find a connected set $K$ that contains $\mathcal{C}$ whose length is minimal, namely one looks for
\begin{equation}\label{ste}
\inf  \{\Ha^1(K) : K \subset \R^2, \mbox{ connected and such that } \mathcal{C}
\subset K\}\,.
\end{equation}
This problem has a long history
(for a detailed presentation we refer to~\cite{IvTuz})
and existence of minimizers is known even
in more general ambients, like suitable metric spaces~\cite{PaSt}.
In particular, whenever the problem is set in $\mathbb{R}^n$,
it is well known that minimizers are finite union of segments that meet 
at triple junctions forming angles of $120$
degrees~\cite{IvTuz}. 
Due to their clear relevance in the problem
we give a name to the elements of this precise class of networks: we call them \emph{minimal networks}.
As the number of points of $\mathcal{C}$
increases, 
the number of configurations that are candidate minimizers 
rapidly increase.
Hence, not only minimizers may not be unique, but identifying them is by no means an easy task, also in 
terms of efficient algorithms
(the Steiner problem is classified as NP-hard
from a computational point of view).

A possible tool to  validate the minimality of a certain candidate is the notion of \emph{calibration}.
The classical concept of calibration for a $d$-dimensional oriented manifold
$M$ in $\mathbb{R}^n$ is a closed
$d$-form $\omega$ such that $\vert\omega\vert\leq 1$ and $\langle \omega, v\rangle = 1$
for every tangent vector $v$ to $M$. 
The existence of a calibration $\omega$ for $M$
is a sufficient condition for $M$ to be area minimizing 
in its homology class. Indeed, let $N$ be a 
$d$-dimensional oriented manifold with the same boundary of 
$M$. Using the conditions satisfied
by $\omega$ and Stokes' theorem, we get
\begin{equation*}
\mbox{Vol}(M) = \int_{M} \omega = \int_N \omega \leq \mbox{Vol}(N)\, ,
\end{equation*}
as desired.

The definition of calibration for minimal surfaces does
not work directly for the Steiner problem
because neither the competitors nor the minimizers of the problem 
admit an orientation which is compatible with their boundary.
However there have been several successful adaptations of the notion: starting from the paired calibrations
by Lawlor and Morgan~\cite{lawmor} 
and their suitable generalizations~\cite{CaPl20,CaPl},
passing to calibrations for integral currents with coefficient
in groups~\cite{MaMa,MaOuVe}, rank-one tensor valued measures~\cite{BoOrOu18,BoOrOu21}, and calibrations to study clusters with multiplicities~\cite{Mor}.
We also mention an adaptation of calibrations to an 
evolution setting presented in~\cite{FiHeLaSi}, 
in this case calibrations are useful to show 
weak-strong uniqueness of $BV$ solutions to the 
multi-phase mean curvature flow.

\medskip

The aim of this note is to prove minimizing properties of minimal networks in $\R^2$ via suitable notions of calibrations. We refer to \cref{sec:Networks} for the basic definitions about networks. Roughly speaking, a network is identified by a compact connected graph $G$ with $1$-dimensional edges and by an immersion $\Gamma:G\to \R^2$, see \cref{defgraph}, \cref{network}. A minimal network $\Gamma_*:G\to \R^2$ is a network with straight edges ending either at endpoints or at junctions of order $3$ forming angles equal to $\tfrac23\pi$, see \cref{def:TripleJunctionsRegularMinimal}.
We shall see that minimal networks enjoy minimizing properties for the \emph{length functional}, defined as the sum of the lengths of each edge. Essentially due to the possible presence of loops, minimal networks
may not be strictly stable critical points of the length functional, nevertheless it turns out that they minimize the length functional among several classes of competitors.

\medskip

In~\cite{MaMa}
Marchese and Massaccesi rephrased the Steiner problem as a mass minimization
problem among rectifiable currents with coefficients in a discrete subgroup $\mathcal{G}$ of $\mathbb{R}^{n-1}$
(with $n$ equal to the number of points to connect)
and their reformulation allows for a natural notion of calibration (see~\cite[Section 3]{MaMa}).
In this paper we prove that we can canonically
associate to any minimal network $\Gamma_*$ a current  $\widehat{T}$
with coefficients in a subgroup of $\mathbb{R}^2$, also producing a calibration for $\widehat{T}$ and hence
 showing that $\widehat{T}$ is mass minimizing among normal currents $T$
with coefficients in $\mathbb{R}^2$ such that $\partial T=\partial \widehat{T}$. We refer to \cref{sec:Currents} for definitions and terminology about currents.

\begin{teo}[{cf. \cref{esibizione-calibrazione-correnti}}]\label{global-correnti}
Let $\Gamma_*:G\to \R^2$ be a minimal network. Then there exists a normed subgroup $(\mathcal{G}, \|\cdot\|)$ of
$\mathbb{R}^2$ such that the following holds.

There exists a $1$-rectifiable current $\widehat{T}$ with coefficients in $\mathcal{G}$ such that
$\mathrm{supp}(\widehat{T})=\Gamma_*$, ${\rm L}(\Gamma_*)=\mathbb{M}(\widehat{T})$, and there exists a calibration
$\omega\in C_c^\infty(\mathbb{R}^2,M^{2\times 2}(\mathbb{R}))$
for $\widehat{T}$ in the sense of \cref{calicurrent}.

In particular, $\widehat{T}$ is a mass
minimizing current among $1$-normal currents
with coefficients in $\mathbb{R}^2$ with the same boundary of $\widehat{T}$.
\end{teo}

In fact, the group $\mathcal{G}$ in the previous theorem will be explicitly exhibited in the proof and we will see that the calibration turns out to be (identified by) the identity matrix. Also, the group $\mathcal{G}$ will be independent of the topology of the minimal network $\Gamma_*$ and of the number of its endpoints. We mention that an analogous argument has been employed in \cite{Wh3, Mor}, of which \cref{global-correnti} can be seen as a particular case; however the proof of \cref{global-correnti} is obtained here by performing a simpler and more explicit construction.

We will also derive consequences of \cref{global-correnti} regarding minimizing properties of minimal networks among several classes of networks. The length of a minimal network $\Gamma_*:G\to\R^2$ can be proved to be less than the one of suitable networks $\Gamma:H\to \R^2$ having (some) endpoints in common with $\Gamma_*$, possibly immersed edges and junctions of higher order. In particular, comparison networks $\Gamma:H\to \R^2$ may have different topology with respect to $\Gamma_*:G\to\R^2$; more precisely, $G$ can be assumed to be (homeomorphic to) a suitable quotient of $H$, see \cref{cor:SottografiQuozientati1}, or, conversely, $H$ can be assumed to be (homeomorphic to) a suitable quotient of $G$, see \cref{cor:SottografiQuozientati2}. These results give applicable comparison results among classes of networks, yielding explicit applications of the more abstract \cref{global-correnti}.

\medskip

The Steiner problem has also been proved to be
equivalent to the so-called minimal partition problem.
A Caccioppoli partition $\mathbf{E}$ of a bounded open Lipschitz set $\Omega$ is a collection of finite perimeter sets $E_1,\ldots,E_n$
sets that are essentially disjoint and that cover $\Omega$.
Given a reference partition 
$\widetilde{\mathbf{E}}=(\widetilde{E}_1,\ldots,\widetilde{E}_n)$ of $\Omega$ we say that $\mathbf{E}=(E_1,\ldots,E_n)$
is a minimizer of the minimal partition problem
if 
$$
\sum_{i=1}^n P(E_i,\Omega)\leq \sum_{i=1}^n P(F_i,\Omega)
$$
among any partition $\mathbf{F}=(F_1,\ldots,F_n)$
with the property that $\tr_\Omega \chi_{\widetilde{E}_i}=\tr_\Omega \chi_{E_i}=\tr_\Omega \chi_{F_i}$, where the latter symbols denote traces on the boundary of $\Omega$
(see for instance~\cite{AmBr1,AmBr2,MoSo}).

In this paper we show that we can associate to a minimal network $\Gamma_*$ an open set $\Omega$ and a partition $\widetilde{\mathbf{E}}=(\widetilde{E}_1,\widetilde{E}_2,\widetilde{E}_3)$ of $\Omega$ such that $\Gamma_*$ coincides with the set of interfaces of $\widetilde{\mathbf{E}}$ in $\Omega$, also producing a suitable calibration for $\widetilde{\mathbf{E}}$ showing that such partition is a minimizer for the minimal partition problem of $\Omega$ among partitions with the same boundary traces. We refer to Section~\ref{Sec:partition} for definitions and terminology on partitions.

\begin{teo}[{cf. \cref{esibizione-calibrazione}}]\label{local-partizioni}
Let $\Gamma_*:G\to \R^2$ be a minimal network such that $\Gamma_*(G)\subset D$ where $D$ is a domain of class $C^1$ homeomorphic to a closed disk, and $\Gamma_*(G) \cap \partial D$ is the set of endpoints of $\Gamma_*$.

Then there exists a bounded open set $\Omega'\subset \R^2$ such that $\Gamma_*(G)\subset \Omega'$, $\Omega\eqdef \Omega'\cap {\rm int\,}(D)$ has Lipschitz boundary, there exists a Caccioppoli partition $\widetilde{\mathbf{E}}=(\widetilde{E}_1,\widetilde{E}_2,\widetilde{E}_3)$ of $\Omega$ such that $\Omega \cap \cup_i\partial\widetilde{E}_i =\Omega \cap \Gamma_*(G)$ and there exists a local paired calibration for $\widetilde{\mathbf{E}}$ in $\Omega$.

In particular, $\widetilde{\mathbf{E}}$ is a minimizer for $\mathcal{P}$ in $\mathcal{A}$, that is the class of partitions having the same trace of $\widetilde{\mathbf{E}}$ on $\partial\Omega$.
\end{teo}

In the above theorem, $\Omega$ is essentially constructed as the intersection of a suitably small tubular neighborhood of $\Gamma_*(G)$ with $D$.
This can be interpreted as a \emph{local} minimality result.
In fact, $\Omega$ \emph{cannot} be taken to be an arbitrary neighborhood of the given minimal network $\Gamma_*$. In Remark~\ref{rem:LocalitaRisultatoPartizioni}, we construct a counterexample to the minimality among partitions in case such tubular neighborhood is too large.

To the best of our knowledge, the tool of 
calibrations has never been used before to 
prove \emph{local} minimality in the framework of the Steiner problem or minimal partitions.

\medskip

\textbf{Addendum.} After this work was completed, a result analogous to \cref{local-partizioni} appeared in \cite{FischerHenLauxSimonCalibrazioni}. More precisely, the authors prove that if a Caccioppoli partition of a given set $D$ is a suitable stationary point of the interface length functional, then it is minimizing for the minimal partition problem of $D$ among partitions sufficiently close in $L^1$ to the stationary one. The previous result is based on a calibration argument similar to the one employed in the present work, but the construction of the calibration is different.

\medskip

\textbf{Organization.} We collect below terminology on networks. In \cref{sec:Currents} we introduce calibrations for currents and we prove \cref{global-correnti}. In \cref{Sec:partition} we define the minimal partition problem, local paired calibrations, and we prove \cref{local-partizioni}.

\medskip

\textbf{Acknowledgements.} The authors are partially supported by the INdAM - GNAMPA Project 2022 CUP \_ E55F22000270001 ``Isoperimetric problems: variational and geometric aspects''.
\[
\text{
\emph{Throughout this paper, we do not identify functions coinciding almost everywhere with respect to a measure.}}
\]

\subsection{Networks}\label{sec:Networks}

For a regular curve $\gamma:[0,1]\to \R^2$ of class $H^2$, define
\[
\tau\eqdef\frac{\gamma'}{|\gamma'|},
\quad
\nu\eqdef {\rm R}(\tau),
\]
the tangent and the normal vector, respectively, where ${\rm R}$ denotes counterclockwise rotation of $\tfrac\pi2$. We define $\de s\eqdef |\gamma'| \de x$ the arclength element and $\partial_s\eqdef |\gamma'|^{-1}\partial_x$ the arclength derivative. The curvature of $\gamma$ is the vector $\kappa \eqdef \partial_s^2 \gamma$.

\begin{defn}\label{defgraph}
Fix $N\in\mathbb{N}$ and let $i\in\{1,\ldots, N\}$, $E_i:=[0,1]\times\{i\}$,
$E:=\bigcup_{i=1}^N E_i$ and 
$V:=\bigcup_{i=1}^N \{0,1\}\times\{i\}$.
Let $\sim$ be an equivalence relation that identifies points of $V$.
A \emph{graph} $G$ is the topological quotient space of $E$ induced by $\sim$, that is
\begin{equation*}
G:=E\Big{/}\sim\,,
\end{equation*}
and we assume that $G$ is connected.

Denoting by $\pi:E\to G$ the projection onto the quotient, an \emph{endpoint} is a point $p \in G$ such that $\pi^{-1}(p) \subset V$ and it is a singleton, a \emph{junction} is a point $m \in G$ such that $\pi^{-1}(m) \subset V$ and it is not a singleton. The \emph{order} of a junction if the cardinality $\sharp \pi^{-1}(m)$. We will always assume that the order of a junction in a graph is greater or equal to $3$.

A \emph{subgraph} $G'\subset G$ is a topological subspace of a graph $G$ which is a graph itself with the structure induced by $G$. More precisely, $G'$ is a graph and there exists a subset $E'\subset E$ such that $G'=E'/_\sim$ where $\sim$ is the same equivalence relation defining $G$.
\end{defn}

\begin{defn}\label{network}
An \emph{immersed network} (or, simply, \emph{network}) is a pair $\mathcal{N}=(G,\Gamma)$
where
\begin{equation*}
\Gamma: G\to \mathbb{R}^2
\end{equation*}
is a continuous map and $G$ is a graph
and each map $\gamma^i:=\Gamma_{\vert E_i}$ is
an immersion of class $C^1$ (up to the boundary).  
\end{defn}

\begin{defn}\label{def:TripleJunctions}
A network $\mathcal{N}=(G,\Gamma)$ is an \emph{immersed triple junctions network} if
it is an immersed network and
each junction of $G$ has order $3$.
\end{defn}

\begin{defn}
Let $\mathcal{N}=(G,\Gamma)$ be a network of class $C^1$ and let $e \in \{0,1\}$. The \emph{inner tangent vector} of a regular curve $\gamma^i$ of $\mathcal{N}$ at $e$ is the vector
\[
(-1)^{e} \frac{(\gamma^i)'(e)}{|(\gamma^i)'(e)|}.
\]
\end{defn}

\begin{defn}\label{def:TripleJunctionsRegularMinimal}
A network $\mathcal{N}=(G,\Gamma)$ is 
\emph{minimal} if each map $\gamma^i:=\Gamma_{\vert E_i}$ is an \emph{embedding} of class $H^2$,
for every $i\ne j$ the curves $\gamma^i$
and $\gamma^j$ do not intersect in their interior, and $\pi(0,i)\neq \pi(1,i)$ for any $i$.
Moreover each junction of $G$ has order $3$
and the sum of the inner tangent vectors 
at a junction is zero.
Furthermore
the curvature of the parametrization of each edge is identically zero, i.e., each $\gamma^i$ is the embedding of a straight segment.
\end{defn}

We shall usually denote a network by directly writing the map $\Gamma:G\to \R^2$. Moreover, 
with a little abuse of terminology, we shall employ the words junctions and endpoints also referring to their images in $\R^2$.

\begin{defn}
Given an immersed network $(G,\Gamma)$
we denote by $\ell (\gamma^i)$ the length of the curve $\gamma^i$.
The \emph{length} of the network $\Gamma$ is
\begin{equation*}
L(\Gamma):=\sum_{i=1}^N \ell (\gamma^i)\,.
\end{equation*}
\end{defn}

By computing the first variation of the length
functional, one easily gets two necessary conditions that a network has to satisfy to be a critical point
of $L$: each curve of the network is a straight segment and the inner tangent vectors of the curves meeting at a junction sum up to zero.
Hence minimal networks are critical points of the length
functional. 
However also networks with junctions of order higher than three may happen to be critical points.

\section{Minimality among currents with coefficients in a group}\label{sec:Currents}

In this section we give a summary of the theory of currents
with coefficients in a group presented in a simplified setting that is convenient for our purposes.
For further details we refer for instance 
to~\cite{Fl,Wh1,Wh2,MaMa}.

Let $k\in\{0,1\}$.
Consider $\mathbb{R}^2$ endowed with a norm 
$\Vert\cdot\Vert$ and 
denote by $\Vert\cdot\Vert_{\ast}$ the corresponding dual norm.
We denote by $\Lambda_k(\mathbb{R}^2)$ the vector space of $k$-vectors in $\R^2$. In particular, $\Lambda_0(\R^2)$ coincides with $\R$ and $\Lambda_1(\R^2)$ is just $\R^2$.

\begin{defn}[$k$-covector with values in $\mathbb{R}^2$]
A \emph{$k$-covector with values in $\mathbb{R}^2$} is a linear map $\omega:\Lambda_k(\mathbb{R}^2)\to\mathbb{R}^2$.
We denote by $\Lambda^k_{2}(\mathbb{R}^2)$
the space of $k$-covectors with values in $\mathbb{R}^2$.

We define the \emph{comass} of a covector 
$\omega\in \Lambda^k_{2}(\mathbb{R}^2)$ 
\begin{equation*}
\vert \omega\vert_{com}:=\sup\left\lbrace
\Vert \omega(\tau)\Vert_\ast\,:\;\tau\in \Lambda_k(\mathbb{R}^2)\;\text{with}\, 
\vert \tau\vert\leq1
\right\rbrace\,,
\end{equation*}
where $|\tau|$ is the norm of a $k$-vector with respect to the Euclidean norm.
\end{defn}

Observe that $0$-covectors with values in $\R^2$ are linear maps $\omega:\R\to \R^2$, while $1$-covectors with values in $\R^2$ coincide with linear maps from $\R^2$ to $\R^2$.

We remark that, since $k\in\{0,1\}$ in the previous definition, $\tau \in \Lambda_k(\R^2)$ is automatically a simple $k$-vector, and thus the previous definition of comass coincides with the usual one considered in the theory of currents.

\begin{defn}[$k$-form with values in $\mathbb{R}^2$]
A \emph{$k$-form with values in $\R^2$} is a function $\omega:\R^2 \to \Lambda^k_2(\R^2)$ with compact support such that $x\mapsto \omega(x)(\tau)$ is smooth for any $\tau \in \Lambda_k(\R^2)$. We denote by
$C^\infty_c(\mathbb{R}^2, \Lambda^k_2(\mathbb{R}^2))$ the space of $k$-forms with values in $\mathbb{R}^2$.

The \emph{comass} of $\omega\in C^\infty_c(\mathbb{R}^2, \Lambda^k_2(\mathbb{R}^2))$ is defined by 
\begin{equation*}
\Vert\omega\Vert_{com}:=\sup_{x\in\mathbb{R}^2} \vert\omega(x)\vert_{com}\,.
\end{equation*}

The space $C^\infty_c(\mathbb{R}^2, \Lambda^k_2(\mathbb{R}^2))$ is endowed with the following notion of convergence: we say that $\omega_n \in C^\infty_c(\mathbb{R}^2, \Lambda^k_2(\mathbb{R}^2))$ converges to $\omega \in C^\infty_c(\mathbb{R}^2, \Lambda^k_2(\mathbb{R}^2))$ if there exists a compact set $K\subset \R^2$ such that the support of $\omega_n$ is contained in $K$ for any $n$ and $\omega_n(\cdot)(\tau)$ converges to $\omega(\cdot)(\tau)$ in $C^m(K)$ for any $m$ for any $\tau \in \Lambda_k(\R^2)$.
\end{defn}

A form $\omega \in C^\infty_c(\mathbb{R}^2, \Lambda^k_2(\mathbb{R}^2))$ is identified by a couple $(\omega_1,\omega_2)$ where $\omega_i :\R^2\to \Lambda^k(\R^2)$ is a standard $k$-form with compact support, in the sense that $\omega(x)(\tau) = (\omega_1(x)(\tau) ,\omega_2(x)(\tau) )$ for any $\tau \in \Lambda_k(\R^2)$ and $x \in \R^2$.
We define the differential $d \omega$ of $\omega \in C^\infty_c(\mathbb{R}^2, \Lambda^k_2(\mathbb{R}^2))$ component-wise by $d \omega\eqdef (d \omega_1,d \omega_2)$.

We can now define also $k$-currents with coefficients in $\R^2$.
\begin{defn}[$k$-current with coefficients in $\mathbb{R}^2$]
A \emph{$k$-current with coefficients in $\mathbb{R}^2$}
is a linear map
\begin{equation*}
T:C^\infty_c(\mathbb{R}^2, \Lambda^k_{2}(\mathbb{R}^2))\to \mathbb{R}\,,
\end{equation*}
that is continuous with respect to convergence in $C^\infty_c(\mathbb{R}^2, \Lambda^k_{2}(\mathbb{R}^2))$, i.e., if $\omega_n\to \omega$ with respect to convergence in $C^\infty_c(\mathbb{R}^2, \Lambda^k_{2}(\mathbb{R}^2))$ then $T(\omega_n)\to T(\omega)$.

The \emph{boundary} of a 
$1$-current $T$ with coefficients in $\mathbb{R}^2$  is the $0$-current with coefficients in $\R^2$
defined by
\begin{equation*}
\partial T(\omega):=T(d\omega)\,,\quad \forall \omega \in C^\infty_c(\mathbb{R}^2, \Lambda^0_{2}(\mathbb{R}^2)).
\end{equation*}
Given $T$ a $k$-current with coefficients in $\mathbb{R}^2$, its \emph{mass} is
\begin{equation*}
\mathbb{M}(T):=\sup\left\lbrace T(\omega)\,:\;\omega\in 
C^\infty_c(\mathbb{R}^2, \Lambda^k_{2}(\mathbb{R}^2))
\;\text{with}\,\Vert\omega\Vert_{com}\leq 1
\right\rbrace\,.
\end{equation*}
A $1$-current $T$ with coefficients in $\mathbb{R}^2$ is said to be \emph{normal} if $\mathbb{M}(T)<\infty$ and $\mathbb{M}(\partial T)<\infty$.
\end{defn}

We are now able to define the object of our main interest for our purposes.

\begin{defn}[$1$-rectifiable current with coefficients in $\mathcal{G}$]
Let $\Sigma\subset \R^2$ be a $1$-rectifiable set. An \emph{orientation} $\tau$ on $\Sigma$ is a measurable map $\tau:\Sigma\to \R^2$ such that $\tau(x) \in T_x\Sigma$ and $\vert\tau(x)\vert=1$ for $\Ha^1$-a.e. $x\in \Sigma$. Let $\mathcal{G}$ be a discrete subgroup of $(\R^2,+)$. A \emph{$\mathcal{G}$-valued multiplicity function} $\theta$ on $\Sigma$ is a function in $L^1_{\rm loc}(\Ha^1\res \Sigma;\mathcal{G})$.

A $1$-current $T$ \emph{is rectifiable with coefficients in $\mathcal{G}$}
if there exist a $1$-rectifiable set $\Sigma\subset \R^2$, an orientation $\tau$, and a $\mathcal{G}$-valued multiplicity function $\theta$ on $\Sigma$ such that, recalling that $\R^2\equiv \Lambda_1(\R^2)$, there holds
\begin{equation*}
T(\omega)=\int_{\Sigma} \left\langle
\omega(x)(\tau(x)), \theta(x)\right\rangle\,\mathrm{d}\mathcal{H}^1\,,
\end{equation*}
for any $\omega \in C^\infty_c(\mathbb{R}^2, \Lambda^1_{2}(\mathbb{R}^2))$,  where $\scal{\cdot,\cdot}$ denotes the usual Euclidean scalar product on $\R^2$.

In analogy with the usual $1$-rectifiable currents, 
a $1$-rectifiable current with coefficients in $\mathcal{G}$
will be denoted by the triple $T=[\Sigma, \tau, \theta]$.
\end{defn}

Thanks to the above representation, 
if $T=[\Sigma, \tau, \theta]$ is a 
$1$-rectifiable current with coefficients in $\mathcal{G}$
one can write its mass as 
\begin{equation*}
\mathbb{M}(T)=\int_{\Sigma}\Vert \theta(x)\Vert\,\mathrm{d}\mathcal{H}^1\,.
\end{equation*}

\begin{rem}[Currents with multiplicity $g$ induced by an immersion]\label{rem:BoundaryImmersionSegment}
%
If $\gamma:[0,1]\to \R^2$ is a Lipschitz immersion 
and $g \in \R^2$ is a fixed vector, the immersion induces a $1$-rectifiable current $T=[\gamma([0,1]), \tau,\theta]$ by choosing $\tau$ and $\theta$ as follows. Let $\widetilde\tau(x)\eqdef \sum_{p\in\gamma^{-1}(x)} \gamma'(p)/|\gamma'(p)|$ and $\theta(x)\eqdef |\widetilde\tau(x)|\,g$ for any $x\in \gamma([0,1])$ such that $\gamma$ is differentiable at any $p\in\gamma^{-1}(x)$ with $|\gamma'(p)|>0$, and then $\tau(x)\eqdef \widetilde\tau(x)/|\widetilde\tau(x)|$ for any $x$ such that also $|\widetilde\tau(x)|\neq 0$ ($\tau$ and $\theta$ are defined arbitrarily elsewhere).\\
The area formula immediately yields $T(\omega)=\int_0^1 \scal{\omega(\gamma(t))(\gamma'(t)),g} \de t$ for any $1$-form $\omega$ with values in $\R^2$. In particular $\partial T = g \delta_{\gamma(1)} - g \delta_{\gamma(0)}$.
\end{rem}

As noticed above, the space of $1$-covectors with values in $\R^2$ is the space of linear maps from $\R^2$ to $\R^2$, hence it is isomorphic to the space of matrices $M^{2\times 2}(\R)$. From now on, we shall make use of this identification without further mention 
and we will denote the set of $1$-forms by $C_c^\infty(\mathbb{R}^2,M^{2\times 2}(\mathbb{R}))$.
Therefore, for $\omega \in C_c^\infty(\mathbb{R}^2,M^{2\times 2}(\mathbb{R}))$ we shall write
\begin{equation*}
\omega=
\begin{bmatrix}
\omega_1(x) \\
\omega_2(x) \\
\end{bmatrix}\,,
\end{equation*}
where $\omega_i = (\omega_{i,1},\omega_{i,2}) : \R^2 \rightarrow \R^2$ is smooth with compact support and identifies the standard $1$-form $\omega_{i,1}\d x^1 + \omega_{i,2} \d x^2$.

We recall here the notion of calibrations for 
rectifiable currents with coefficients in a group 
introduced by Marchese and Massaccesi in~\cite{MaMa}. 

\begin{defn}[Calibration for $1$-rectifiable currents]\label{calicurrent}
Let $T=[\Sigma, \tau, \theta]$ be a 
$1$-rectifiable current with coefficients in $\mathcal{G}$ and 
$\omega\in C_c^\infty(\mathbb{R}^2,M^{2\times 2}(\mathbb{R}))$. 
Then 
$\omega$ is a calibration for $T$ if
\begin{itemize}
\item[(i)] $d\omega = 0$;
\item[(ii)] $\|\omega\|_{com} \leq 1$;
\item[(iii)]  $\langle \omega(x)(\tau(x)),\theta(x)\rangle 
= \|\theta(x)\|$ for $\Ha^1$-a.e. $x \in \Sigma$.
\end{itemize}
\end{defn}

We recall from \cite{MaMa} how this notion of calibration for currents implies minimality properties for the mass of the calibrated current current.

For a given set of points in the plane
$p_1,\ldots,p_n\in\mathbb{R}^2$, let $B$ be a
$0$-current of the form 
$$
B:=c_1\delta_{p_1}+\ldots+c_n\delta_{p_n}\quad\text{with}\;c_i\in \mathcal{G}\,,
$$
i.e., $B(f) = \sum_{i=1}^n \scal{c_i, f(p_i)}$ for any $0$-form $f$ with values in $\R^2$.

\begin{rem}\label{rem:BoundarySumZero}
It is easily checked that $B=c_1\delta_{p_1}+\ldots+c_n\delta_{p_n}$ is the boundary of a $1$-rectifiable current with bounded support with coefficients in $\mathcal{G}$ if and only if $\sum_{i=1}^n c_i=0$.

Indeed, if $B=\partial T$, let $\omega=v \chi(x)$ be the $0$-form with values in $\R^2$ where $v \in \R^2$ is fixed and $\chi \in C^\infty_c(\R^2)$ equals $1$ in a neighborhood of the support of $T$; hence $0=T(d\omega)=B(\omega)= \scal{v, \sum_{i=1}^n c_i}$. Arbitrariness of $v$ implies $\sum_{i=1}^n c_i=0$. On the other hand, if $\sum_{i=1}^n c_i=0$, let $q \neq p_i$ for any $i$; hence it suffices to take $T=\sum_{i=1}^n T_i$ where $T_i$ is the current induced by a $C^1$ embedding $\gamma_i:[0,1]\to \R^2$ from $q$ to $p_i$ 
endowed with multiplicity $c_i$.
\end{rem}

We define the 
classes
\begin{align*}
    \mathcal{C}_1:=\left\lbrace T\st T\mbox{ is a }  1\text{-rectifiable 
currents with coefficients in }\mathcal{G}, \ \partial T=B\right\rbrace\,.\\
    \mathcal{C}_2:=\left\lbrace T\st T\mbox{ is a }  1\text{-normal
currents with coefficients in }\mathbb{R}^2, \ \partial T=B\right\rbrace\,.
\end{align*}
Obviously $\mathcal{C}_1\subset\mathcal{C}_2$.

In case $\mathcal{C}_1$ is nonempty and contains a current $\widehat{T}$ with a calibration $\omega$, then $\widehat{T}$ solves a mass minimization problem as stated in the next proposition.

\begin{prop}[{\cite[Proposition 3.2]{MaMa}}]
In the notation above, suppose that $\omega\in C_c^\infty(\mathbb{R}^2,M^{2\times 2}(\mathbb{R}))$ is a calibration
for some $\widehat{T}\in\mathcal{C}_1$. Then 
\begin{equation*}
    \mathbb{M}(\widehat{T})\leq \mathbb{M}(T)
\end{equation*}
for every $T\in\mathcal{C}_2$.
\end{prop}

\medskip
From now on, let $\{e_1,e_2\}$ be the canonical basis of $\mathbb{R}^2$.
We define $g_1:=e_1$, $g_2:=(-\frac{1}{2},-\frac{\sqrt{3}}{2})$ and 
$g_3:=-g_1-g_2$.
We choose a norm $\Vert \cdot\Vert$ on $\mathbb{R}^2$ such that $\Vert g_1\Vert=\Vert g_2\Vert=\Vert g_1+g_2\Vert=1$, the unit ball with respect to $\Vert \cdot\Vert$ is the regular hexagon in \cref{fig:norm}, and $\Vert \cdot\Vert$ is defined on the rest of $\R^2$ by homogeneity. We define $\mathcal{G}$
to be the discrete group generated by $g_1$ and $g_2$ with respect to addition.

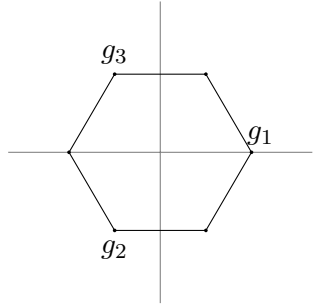
\begin{figure}[H]
    \centering
\begin{tikzpicture}[scale=0.4]
\draw[black!50!white]
(0,-5)--(0,5)
(-5,0)--(5,0);
\draw[black]
(3,0)--(1.5, 2.59)--(-1.5, 2.59)
--(-3,0)--(-1.5, -2.59)--(1.5, -2.59)--(3,0);
\fill[black](-1.5, -2.59)circle (1.7pt);
\fill[black](-1.5, 2.59)circle (1.7pt);
\fill[black](3,0)circle (1.7pt);
\fill[black](-3,0)circle (1.7pt);
\fill[black](1.5, 2.59)circle (1.7pt);
\fill[black](1.5, -2.59)circle (1.7pt);
\path
(3.3, -0.1)node[above]{$g_1$}
(-1.5, -2.59)node[below]{$g_2$}
(-1.5, 2.59)node[above]{$g_3$};
\end{tikzpicture}
    \caption{The unit ball in the norm $\Vert\cdot\Vert$.}
    \label{fig:norm}
\end{figure}

\begin{teo}\label{esibizione-calibrazione-correnti}
Let $\Gamma_*:G\to \R^2$ be a minimal network
with endpoints $p_1,\ldots,p_n \in \R^2$. Let $\mathcal{G}$, $g_1,g_2,g_3$, and $\|\cdot\|$ be as above.

Then there exists a $1$-rectifiable current $\widehat{T}$ with coefficients in $\mathcal{G}$ with boundary $B\eqdef\partial\widehat{T}=c_1\delta_{p_1}+\ldots+c_n\delta_{p_n}$
such that $c_i\in\{\pm g_1,\pm g_2, \pm g_3\}$,
$\mathrm{supp}(\widehat{T})=\Gamma_*$, ${\rm L}(\Gamma_*)=\mathbb{M}(\widehat{T})$, and there exists a calibration
$\omega\in C_c^\infty(\mathbb{R}^2,M^{2\times 2}(\mathbb{R}))$
for $\widehat{T}$.

In particular, $\widehat{T}$ is a mass
minimizing current among $1$-normal currents
with coefficients in $\mathbb{R}^2$ with boundary
$B$.
\end{teo}
\begin{proof}
Up to 
translations we can fix one of the endpoints of 
$\Gamma_*$ to be the origin of $\mathbb{R}^2$
and we can rotate
$\Gamma_*$ so that all its straight edges are parallel to either $g_1$, $g_2$, or $g_3$.
Then the desired $1$-rectifiable current $\widehat{T}=[\Sigma,\tau,\theta]$
with coefficients in $\mathcal{G}$ is defined as follows: $\Sigma=\Gamma_*(G)$ is the $1$-rectifiable set,
the orientation $\tau$
and the multiplicity $\theta$ are constant 
in the interior of each straight segment and we set $\tau(x)=\theta(x)=g_i$ if and only if $x$ is an interior point of a straight edge of $\Gamma_*$ parallel to $g_i$ ($\tau$ and $\theta$ are defined arbitrarily at endpoints and junctions).

It is immediately checked that $B\eqdef \partial \widehat{T}$ has the desired form, as no boundary is generated at triple junctions. Moreover $\mathrm{supp}(\widehat{T})=\Gamma_*$ and ${\rm L}(\Gamma_*)=\mathbb{M}(\widehat{T})$.

Finally we claim that the identity matrix 
$$
\omega=\begin{bmatrix}
1 & 0 \\
0 & 1\\
\end{bmatrix},
$$
identifying a $1$-form in $C_c^\infty(\mathbb{R}^2,M^{2\times 2}(\mathbb{R}))$, is a calibration for $\widehat{T}$. Indeed $d \omega=0$ trivially. Next, to show that $\|\omega\|_{com}\le 1$ we notice that the unit ball 
of the norm of the group is convex and that its
extreme points are $\pm g_1,\pm g_2,\pm g_3$ and hence it is sufficient
to estimate $\left\langle\omega (v(x)),\cdot\right\rangle$ against $\pm g_1,\pm g_2,\pm g_3$, where $v(x)$
is generic unit vector $v(x)=(\cos\alpha(x),\sin\alpha(x))$. So we have
 
 \begin{align*}
  \vert \left\langle\omega ( v(x)), g_1\right\rangle \vert &=  
  \bigg\vert
   \left\langle 
\begin{bmatrix}
1 & 0 \\
0 & 1\\
\end{bmatrix}
\begin{bmatrix}
\cos\alpha(x) \\
\sin\alpha(x) \\
\end{bmatrix},\begin{bmatrix}
 1  \\
0 \\
\end{bmatrix}\right\rangle \bigg\vert =
\vert  \cos\alpha(x)\vert\leq 1\,,
\end{align*}
and similarly $\vert \left\langle\omega ( v(x)),g_2\right\rangle \vert = \vert \sin(\alpha(x)+\pi/6)\vert \leq 1$, and $\vert    \left\langle\omega (v(x)),g_3\right\rangle\vert =\vert \sin(\alpha(x)-\pi/6) \vert \leq 1$.

To conclude we check that $\scal{\omega(\tau(x)),\theta(x)}=\|\theta(x)\|$ $\Ha^1$-a.e. on $\Sigma$.  But this is immediate as 
$$
\scal{\omega(\tau(x)),\theta(x)}=\scal{\tau(x),\theta(x)}= \scal{\theta(x),\theta(x)}= 1 = \|\theta(x)\|
$$ by definition of the orientation and of the norm $\|\cdot\|$.
\end{proof}

Using \cref{esibizione-calibrazione-correnti} we directly derive a proof of the fact that minimal networks minimize length among competitor networks having the same topology.

\begin{cor}\label{stessa-topologia}
Let $\Gamma_*:G\to \R^2$ be a minimal network. Then ${\rm L}(\Gamma_*)\le {\rm L}(\Gamma)$ for any immersed triple junctions network $\Gamma:G\to \R^2$ having the same endpoints of $\Gamma_*$.
\end{cor}

\begin{proof}
Let $\widehat{T}$ be given by \cref{esibizione-calibrazione-correnti}, let $B\eqdef \partial \widehat{T}$.
The claim follows from \cref{esibizione-calibrazione-correnti} if we can show that given an immersed triple junctions network $\Gamma:G\to \R^2$ having the same endpoints of $\Gamma_*$, there exists a $1$-rectifiable current with coefficient in $\mathcal{G}$ and boundary $\partial T= B$ such that $\mathbb{M}(T)\le{\rm L}(\Gamma)$.

Recalling the construction of $\widehat{T}=[\Gamma_*(G),\tau,\theta]$ in the proof of \cref{esibizione-calibrazione-correnti}, the multiplicity $\theta$ is constant along each edge of $\Gamma_*$. Hence for any edge $E_i$ of the graph $G$ we can associate $g_{E_i} \in \{g_1,g_2,g_3\}$ such that $\theta(\Gamma_*(p)) = g_{E_i}$ for $\Ha^1$-a.e. $p \in E_i$. Moreover, up to inverting the orientation of the edges of $G$, we can assume that $\tau(\Gamma_*(p))=\tau_{\Gamma_*|_{E_i}}(p)$ for any $p$ in the interior of $E_i$ and any $i$, i.e., the orientation of $\widehat{T}$ is the one induced by the immersions $\Gamma_*|_{E_i}$.


So we can define the desired current by taking $T\eqdef[\Gamma(G), \tau_\Gamma, \theta_\Gamma]$ as follows. For any edge $E_i$ of $G$ let $T_{E_i}\eqdef [ \Gamma|_{E_i}(E_i), \tau_i, \theta_i]$ be the current induced by $\Gamma|_{E_i}$ as in \cref{rem:BoundaryImmersionSegment} taking $g=g_{E_i}$. Hence define $T\eqdef \sum_{E_i} T_{E_i}$.

Since $\Gamma_*$ and $\Gamma$ have the same domain graph $G$, whose edge have a fixed orientation, recalling \cref{rem:BoundaryImmersionSegment} it follows that $\partial T = \partial \widehat{T}=B$. Moreover, since $\|\theta_\Gamma(x)\| \le \sum_{E_i} \sharp (\Gamma|_{E_i})^{-1}(x) \, \| g_{E_i}\| = \sharp \Gamma^{-1}(x)$, it follows that $\mathbb{M}(T)\le {\rm L}(\Gamma)$.
\end{proof}

\begin{rem}\label{rem:ComparisonLips1}
By approximation, the result in \cref{stessa-topologia} clearly holds also in case the maps $\Gamma|_{E_i}$ of the comparison networks are just Lipschitz immersions, i.e., Lipschitz maps with almost everywhere non-vanishing derivative $\left(\Gamma|_{E_i}\right)'$.
\end{rem}

We now derive further minimizing properties of minimal networks among competitors with possibly different topologies.

Consider a graph $G$ and 
another graph $H$ that contains a copy of the graph $G$.
The next corollary roughly states that if 
we consider a minimal network 
$\Gamma_*:G\to \R^2$ 
and a network $\Gamma:H\to \R^2$ with the same endpoints
of $\Gamma_*$, then the length of  $\Gamma$ is no less than the one of $\Gamma_*$.\\
We observe that some topological assumption of this kind is also necessary for the length of a given $\Gamma_*$ to be minimizing. Otherwise, given a fixed minimal network $\Gamma_*:G\to \R^2$ whose image contains a cycle, just by deleting an interior segment of an edge of the cycle one gets a strictly shorter minimal network (with two additional endpoints) as depicted in Figure~\ref{Fig:shorter}.

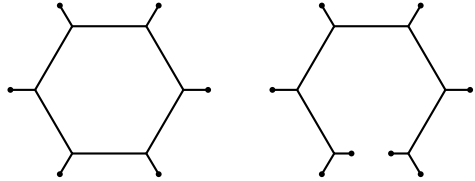
\begin{figure}[H]
	\begin{center}
	\begin{tikzpicture}[scale=1.3]
\draw[thick]
(-0.5,-0.86)--(-0.375,-0.649)
(-1,0)-- (-0.75,0)
(-0.5,0.86)--(-0.375,0.649)
(0.5,-0.86)--(0.375,-0.649)
(1,0)-- (0.75,0)
(0.375,0.649)--(0.5,0.86) ;
\draw[thick]
(-0.375,-0.649)--(-0.75,0)--(-0.375,0.649)--
(0.375,0.649)--(0.75,0)--(0.375,-0.649)--(-0.375,-0.649);
\fill[black](1,0) circle (0.85pt);  
\fill[black](-1,0) circle (0.85pt);  
\fill[black](0.5,-0.86) circle (0.85pt);  
\fill[black](-0.5,-0.86) circle (0.85pt);  
\fill[black](0.5,0.86) circle (0.85pt);  
\fill[black](-0.5,0.86) circle (0.85pt);  
\end{tikzpicture}\qquad
\begin{tikzpicture}[scale=1.3]
\draw[thick]
(-0.5,-0.86)--(-0.375,-0.649)
(-1,0)-- (-0.75,0)
(-0.5,0.86)--(-0.375,0.649)
(0.5,-0.86)--(0.375,-0.649)
(1,0)-- (0.75,0)
(0.375,0.649)--(0.5,0.86) ;
\draw[thick]
(-0.2,-0.649)--(-0.375,-0.649)--(-0.75,0)
--(-0.375,0.649)--(0.375,0.649)--(0.75,0)--(0.375,-0.649)--(0.2,-0.649);
\fill[black](-0.2,-0.649) circle (0.85pt);
\fill[black](0.2,-0.649) circle (0.85pt);
\fill[black](1,0) circle (0.85pt);  
\fill[black](-1,0) circle (0.85pt);  
\fill[black](0.5,-0.86) circle (0.85pt);  
\fill[black](-0.5,-0.86) circle (0.85pt);  
\fill[black](0.5,0.86) circle (0.85pt);  
\fill[black](-0.5,0.86) circle (0.85pt);  
\end{tikzpicture}
	\end{center}
	\caption{Left: A minimal network $\Gamma_*:G\to \R^2$. Right: A network $\Gamma:H\to \R^2$, the graph $H$ does not contain a homeomorphic copy of $G$ and clearly $L(\Gamma)<L(\Gamma_*)$.
	}\label{Fig:shorter}
\end{figure}

\begin{cor}\label{cor:ContieneCopiaG}
Let $\Gamma_*:G\to \R^2$ be a minimal network. Let $\Gamma:H\to \R^2$ be an immersed network such that there exist a subset $G'\subset H$  and a homeomorphism $f:G\to G'$ such that $\Gamma_*(q) = \Gamma(f(q))$ for any endpoint $q$ of $G$.
Then ${\rm L}(\Gamma_*)\le {\rm L}(\Gamma)$.
\end{cor}

\begin{proof}
For any edge $E_i$ of $G$, let $J_i$ be the set of junctions of $H$ contained in $f(E_i)$.
Up to homeomorphism, we can assume that the restriction of $f$ on $E_i\setminus f^{-1}(J_i)$ is a smooth diffeomorphism between intervals. Hence for any edge $E_i$ of $G$, the map $\Gamma\circ f|_{E_i}$ is Lipschitz with almost everywhere non-vanishing derivative.
Therefore the map $\Gamma'\eqdef \Gamma\circ f:G\to \R^2$ defines a (Lipschitz regular) immersed triple junctions network with same endpoints of $\Gamma_*$. Hence \cref{stessa-topologia} and \cref{rem:ComparisonLips1} apply and we get ${\rm L}(\Gamma_*) \le {\rm L}(\Gamma') \le {\rm L}(\Gamma)$.
\end{proof}

Exploiting the previous corollary, it is possible to compare the length of a minimal network $\Gamma_*:G\to \R^2$ with the length of a suitable immersed network $\Gamma:H\to \R^2$ possibly having different topology. In the next statement, we consider comparison networks $\Gamma:H\to \R^2$ having ``richer topology'', in the sense that $G$ is assumed to be homeomorphic to a topological quotient of $H$. In particular, $H$ may have more edges and endpoints than $G$ (see \cref{Fig:competitor-richer-top}).

\begin{cor}\label{cor:SottografiQuozientati1}
Let $\Gamma_*:G\to \R^2$ be a minimal network. Let $\Gamma:H\to \R^2$ be an immersed network such that
\begin{enumerate}[label={\normalfont{\arabic*)}}]
    \item there exist connected pairwise disjoint subgraphs $H_i\subset H$ such that, letting $\overline{H}\eqdef H/_\sim$ the quotient space that identifies each $H_i$ with a point, there exists a homeomorphism $F: G \to  \overline{H}$;
    
    \item $\pi^{-1}(F(q)) \cap \cup_i H_i = \emptyset$ for any endpoint $q$ of $G$, where $\pi:H\to \overline{H}$ is the natural projection;
    
    \item for any endpoint $q$ of $G$ there holds $\Gamma_*(q)= \Gamma(F(q))$.\footnote{$\Gamma(F(q))$ is well defined since, by 2), $\pi^{-1}(F(q))$ is a singleton for any endpoint $q$ of $G$.}
\end{enumerate}
Then ${\rm L}(\Gamma_*)\le {\rm L}(\Gamma)$.
\end{cor}

\begin{proof}
We want to show that there exists a subset $G'\subset H$ homeomorphic to $G$. In this way, employing also 3), applying \cref{cor:ContieneCopiaG} we get the result.

We are going to define $G'$ by selecting suitable paths contained in the subgraphs $H_i$ in order to join the connected components of $H \setminus \cup_i H_i$ so to get a subset $G'$ homeomorphic to $G$.

By 1) and 2), for any subgraph $H_i$, using the homeomorphism $F^{-1}$, the point $\pi(H_i) \in \overline{H}$ corresponds either to an interior point of an edge $E_j$ of $G$, or to a triple junction $m$ of $G$. We distinguish the two cases.
\begin{itemize}
    \item Suppose that, up to renaming, $\pi(H_i)$ corresponds to an interior point of the edge $E_1$ in $G$. Hence $H_i \cap \overline{H\setminus H_i}$ consists of two points $x,y$ that are endpoints of two (different) edges $L_1,L_2$ of $H$. Since $H_i$ is connected, there exists an embedding $\alpha:[0,1]\to H_i$ connecting $x$ and $y$.
    
    Hence $L_1 \cup \alpha([0,1]) \cup L_2$ is homeomorphic to $E_1$.
    
    \item Suppose that, up to renaming, $\pi(H_i)$ corresponds to a triple junction $m$ in $G$ where the edges $E_1,E_2,E_3$ concur. Hence $H_i \cap \overline{H\setminus H_i}$ consists of three points $a_1,a_2,a_3$ that are endpoints of three (different) edges $L_1,L_2,L_3$ of $H$. Up to renaming, there exist embeddings $\alpha_{12}:[0,1]\to H_i$ from $a_1$ to $a_2$ not passing through $a_3$ and $\alpha_{13}:[0,1]\to H_i$ from $a_1$ to $a_3$ not passing through $a_2$.
    
    Indeed $H_i$ is connected, thus there is an embedding $\sigma:[0,1]\to H_i$ from $a_1$ to $a_2$. If $\sigma$ does not touch $a_3$, by connectedness there is an embedding $\tilde{\sigma}:[0,1]\to H_i$ from $a_3$ to $a_1$. If $a_2 \not \in \tilde{\sigma}([0,1])$, we are done. If $a_2  \in \tilde{\sigma}([0,1])$, then the claim follows by taking $a_2$ in place of $a_1$ and by splitting $\tilde{\sigma}$ into two embeddings. If otherwise $a_3 \in \sigma([0,1])$, then the claim follows by taking $a_3$ in place of $a_1$ and by splitting $\sigma$ into two embeddings.
    
    Therefore there is a junction $w$ of $H$ and times $t_2,t_3 \in (0,1)$ such that $\alpha_{12}(t_2)=\alpha_{13}(t_3)=w$ and $\alpha_{12}((t_2,1]) \cap \alpha_{13}((t_3,1]) = \emptyset$.
    
    Hence $L_1 \cup \alpha_{12}([0,1]) \cup \alpha_{13}((t_3,1]) \cup L_2\cup L_3$ is homeomorphic to $E_1\cup E_2\cup E_3$ in $G$.
\end{itemize}
Performing the selections in the previous items for any $H_i$ we obtain subsets $S_i \subset H_i$ such that $G$ is homeomorphic to $G'\eqdef (H\setminus \cup_i H_i ) \cup \bigcup_i S_i$ via a homeomorphism $f:G \to G'$ such that $\Gamma_*(q)= \Gamma(f(q))$ for any endpoint of $G$, by 3). Hence \cref{cor:ContieneCopiaG} applies and the proof follows.
\end{proof}

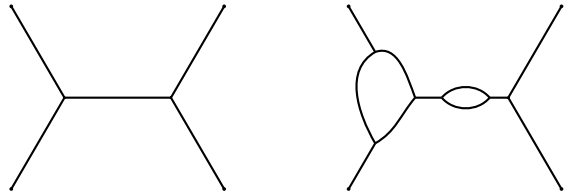
\begin{figure}[H]
\begin{center}
\begin{tikzpicture}[scale=0.35]
\draw[thick]
(0,0)--(-2,-3.46)
(0,0)--(-2,3.46)
(0,0)--(4,0)
(4,0)--(6,-3.46)
(4,0)--(6,3.46);
\fill[black](-2,-3.46) circle (2pt);
\fill[black](-2,3.46) circle (2pt);
\fill[black](6,-3.46) circle (2pt);
\fill[black](6,3.46) circle (2pt);
\end{tikzpicture}\qquad\qquad
\begin{tikzpicture}[scale=0.35]
\draw[thick]
(-1,-1.73)--(-2,-3.46)
(-1,1.73)--(-2,3.46)
(0.5,0)--(1.5,0)
(3.3,0)--(4,0)
(4,0)--(6,-3.46)
(4,0)--(6,3.46);
\draw[thick]
(-1,1.73)to [out=-150, in=120, looseness=1](-1,-1.73)
(-1,1.73)to [out=20, in=110, looseness=1] (0.5,0)
(-1,-1.73)to [out=30, in=-130, looseness=1] (0.5,0)
(1.5,0)to [out=50, in=130, looseness=1] (3.3,0)
(1.5,0)to [out=-50, in=-130, looseness=1] (3.3,0);
\fill[black](-2,-3.46) circle (2pt);
\fill[black](-2,3.46) circle (2pt);
\fill[black](6,-3.46) circle (2pt);
\fill[black](6,3.46) circle (2pt);
\end{tikzpicture}
	\end{center}
	\caption{Left: A minimal network $\Gamma_*:G\to \R^2$. Right: A competitor $\Gamma:H\to \R^2$, fulfilling the hypothesis of Corollary~\ref{cor:SottografiQuozientati1}.}\label{Fig:competitor-richer-top}
\end{figure}

In contrast to \cref{cor:SottografiQuozientati1}, exploiting again the general \cref{esibizione-calibrazione-correnti}, we can further prove that a minimal network $\Gamma_*:G\to \R^2$ is also length-minimizing among suitable immersed networks $\Gamma:H\to \R^2$ having poorer topology. In this case ``poorer topology'' means that $H$ is homeomorphic to a quotient of $G$. In particular, $H$ may have less edges than $G$ (see \cref{Fig:poorer-top}).

\begin{cor}\label{cor:SottografiQuozientati2}
Let $\Gamma_*:G\to \R^2$ be a minimal network. Let $\Gamma:H\to \R^2$ be an immersed network such that
\begin{enumerate}[label={\normalfont{\arabic*)}}]
    \item there exist connected pairwise disjoint subgraphs $G_i\subset G$ such that, letting $\overline{G}\eqdef G/_\sim$ the quotient space that identifies each $G_i$ with a point, there exists a homeomorphism $F:\overline{G}\to H$;
    
    \item for any endpoint $q$ of $G$, there holds that $q \not\in \cup_i G_i$ and $\Gamma_*(q)= \Gamma(F(\pi(q)))$, where $\pi:G\to \overline{G}$ is the natural projection.
\end{enumerate}
Then ${\rm L}(\Gamma_*)\le {\rm L}(\Gamma)$.
\end{cor}

\begin{proof}
Let $\widehat{T}=[\Gamma_*(G),\tau,\theta]$ be given by \cref{esibizione-calibrazione-correnti}. As in the proof of \cref{stessa-topologia}, for any edge $E_j$ of the graph $G$ we can associate $g_{E_j} \in \{g_1,g_2,g_3\}$ such that $\theta(\Gamma_*(p)) = g_{E_j}$ for $\Ha^1$-a.e. $p \in E_j$. Moreover, up to inverting the orientation of the edges of $G$, we can assume that $\tau(\Gamma_*(p))=\tau_{\Gamma_*|_{E_j}}(p)$ for any $p$ in the interior of $E_j$ and any $j$.
Also, for any edge $E_j$, we denote $\widehat{T}_{E_j}\eqdef [ \Gamma_*|_{E_j}, \tau_j,\theta_j]$ where $\tau_j(x)= \tau_{\Gamma_*|_{E_j}}(\Gamma^{-1}(x))$ and $\theta_j(x) = g_{E_j}$ on $\Gamma_*|_{E_j}$ (and zero elsewhere).

For any subgraph $G_i$, consider the current $S_i\eqdef \sum_{E_j \subset G_i} \widehat{T}_{E_j}$. Since by 2) no endpoints of $G$ touch $\cup_i G_i$, then $\partial S_i = \sum_j g^i_j \delta_{m_j^i}$ for some junctions $M_i=\{m_j^i\}_j\subset \R^2$ depending on $i$ and elements of the group $\{g^i_j\}_j$. Hence $\sum_j g^i_j=0$ by \cref{rem:BoundarySumZero}, for any $i$. Moreover, for given $G_i$, we denote by 
$$
\mathscr{E}_i\eqdef \{(e_k,E_k) \st  e_k \in \{0,1\},\text{ $E_k$ edge of $G$ has the endpoint $e_k$ lying in $G_i$} \}\,.
$$
Since $T$ has no boundary at junctions, i.e., $\partial T$ is supported on endpoints, it follows that
\begin{equation}\label{eq:zzNoBoundary}
    \sum_{(e_k,E_k) \in \mathscr{E}_i} (-1)^{1+e_k} g_{E_k}= -\sum_j g^i_j =0,
\end{equation}
for any $i$.

We want to associate to $\Gamma:H\to \R^2$ a suitable current $T$ with $\partial T =\partial \widehat{T}$.
With little abuse of terminology, we will say that $\pi(E_k)$ is an edge of $\overline{G}$ for any edge $E_k$ of $G$ not belonging to $\cup_i G_i$; in particular, for such a $k$, $\pi(E_k)$ is an oriented interval. By 2) we have that $F(\pi(G_i))$ is either a junction of $H$ or an interior point of some edge of $H$. If the latter case happens for some $G_{i_0}$, then $\mathscr{E}_{i_0}= \{ (e_{k_1}, E_{k_1}), (e_{k_2}, E_{k_2})\}$ has two elements. Since $g_{E_j} \in \{g_1,g_2,g_3\}$ for any edge $E_j$ of $G$, then \eqref{eq:zzNoBoundary} implies that $g_{E_{k_1}}=g_{E_{k_2}}$ and $e_{k_1}\neq e_{k_2}$. Hence there is an embedding $\alpha_{k_1,k_2}:[0,1]\to \overline{G}$ whose image is $\pi(E_{k_1})\cup \pi(G_i) \cup \pi(E_{k_2})$ and $\alpha_{k_1,k_2}$ preserves the orientation of $\pi(E_{k_1}), \pi(E_{k_2})$, i.e., the restriction $\alpha_{k_1,k_2}:\alpha_{k_1,k_2}^{-1}(\pi(E_{k_\ell}))\to \pi(E_{k_\ell})$ is an orientation preserving homeomorphism for $\ell=1,2$.

Therefore, up to inverting orientations of edges of $H$, we have that the restriction $F: F^{-1}(H_s)\to H_s$ is orientation preserving for any edge $H_s$ of $H$, i.e., if $\pi(E_k)\subset F^{-1}(H_s)$ and $E_k \cap (\cup_i G_i)=\emptyset$ then the restriction $F: \pi(E_k)\to F(\pi(E_k))\subset H_s$ is orientation preserving. Also, we can associate to any edge $H_s$ of $H$ a group element $g_s \in \{g_1,g_2,g_3\}$ where $g_s=g_{E_k}$ for any edge $E_k$ of $G\setminus (\cup_i G_i)$ with $\pi(E_k) \subset F^{-1}(H_s)$.

Finally we can define the desired current $T=\sum_s T_{H_s}$ where $T_{H_s}\eqdef [ \Gamma|_{H_s}, \sigma_s, \vartheta_s ]$ is the current induced by the immersion $\Gamma|_{H_s}$ as in \cref{rem:BoundaryImmersionSegment} taking $g=g_s$.

By the above observations we have $\partial T_{H_s} = g_s\delta_{\Gamma|_{H_s}(1)} - g_s \delta_{\Gamma|_{H_s}(0)}$. Now junctions of $H$ correspond: either to junctions of $G$ not belonging to $\cup_i G_i$, or to identified graphs $\pi(G_i)$ in $\overline{G}$. Calling $M\eqdef \{ m_i=\Gamma(F(\pi(G_i))) \st F(\pi(G_i)) \text{ is a junction in $H$} \} \subset \R^2$, since $F$ is orientation preserving, we get
\[
\begin{split}
    \partial T 
    &= \sum_s \partial T_{H_s} = \partial \widehat{T} + \sum_{m_i \in M} \sum_{(e_k,E_k) \in \mathscr{E}_i} (-1)^{1+e_k} g_{E_k} m_i \overset{\eqref{eq:zzNoBoundary}}{=} \partial \widehat{T}.
\end{split}
\]
Moreover $\|\sum_k \vartheta_k(x)\| \le \sum_k \sharp (\Gamma|_{H_k})^{-1}(x) \, \| g_{E_k}\| = \sharp \Gamma^{-1}(x)$, thus $\mathbb{M}(T)\le {\rm L}(\Gamma)$. Hence \cref{esibizione-calibrazione-correnti} implies  ${\rm L}(\Gamma_*)\le \mathbb{M}(T)\le {\rm L}(\Gamma)$.

\end{proof}

\begin{figure}[H]
	\begin{center}
	\begin{tikzpicture}[scale=1.3]
\draw[thick]
(-0.5,-0.86)--(-0.375,-0.649)
(-1,0)-- (-0.75,0)
(-0.5,0.86)--(-0.375,0.649)
(0.5,-0.86)--(0.375,-0.649)
(1,0)-- (0.75,0)
(0.375,0.649)--(0.5,0.86) ;
\draw[thick]
(-0.375,-0.649)--(-0.75,0)--(-0.375,0.649)--
(0.375,0.649)--(0.75,0)--(0.375,-0.649)--(-0.375,-0.649);
\fill[black](1,0) circle (0.85pt);  
\fill[black](-1,0) circle (0.85pt);  
\fill[black](0.5,-0.86) circle (0.85pt);  
\fill[black](-0.5,-0.86) circle (0.85pt);  
\fill[black](0.5,0.86) circle (0.85pt);  
\fill[black](-0.5,0.86) circle (0.85pt);  
\path
(-1, 0.7)node[above]{$\Gamma_*$};
\end{tikzpicture}\qquad
\begin{tikzpicture}[scale=1.3]
\draw[thick]
(1,0)--(-1,0)
(-0.5,-0.86) --(0.5,0.86)
(0.5,-0.86) --(-0.5,0.86);
\fill[black](1,0) circle (0.85pt);  
\fill[black](-1,0) circle (0.85pt);  
\fill[black](0.5,-0.86) circle (0.85pt);  
\fill[black](-0.5,-0.86) circle (0.85pt);  
\fill[black](0.5,0.86) circle (0.85pt);  
\fill[black](-0.5,0.86) circle (0.85pt);  
\path
(-1, 0.7)node[above]{$\Gamma$};
\end{tikzpicture}\qquad\qquad\qquad\quad
\begin{tikzpicture}[scale=0.3]
\draw[thick]
(0,0)--(-2,-3.46)
(0,0)--(-2,3.46)
(0,0)--(4,0)
(4,0)--(6,-3.46)
(4,0)--(6,3.46);
\fill[black](-2,-3.46) circle (3.5pt);
\fill[black](-2,3.46) circle (3.5pt);
\fill[black](6,-3.46) circle (3.5pt);
\fill[black](6,3.46) circle (3.5pt);
\path
(-3, 3.3)node[above]{$\Gamma_*$};
\end{tikzpicture}\qquad\qquad
\begin{tikzpicture}[scale=0.3]
\draw[thick]
(-2,3.46)--(6,-3.46)
(-2,-3.46)--(6,3.46);
\fill[black](-2,-3.46) circle (3.5pt);
\fill[black](-2,3.46) circle (3.5pt);
\fill[black](6,-3.46) circle (3.5pt);
\fill[black](6,3.46) circle (3.5pt);
\path
(-3, 3.3)node[above]{$\Gamma$};
\end{tikzpicture}
	\end{center}
	\caption{Two couples of minimal networks $\Gamma_*:G\to\mathbb{R}^2$
	and competitors $\Gamma:H\to\mathbb{R}^2$
	fulfilling the hypothesis of Corollary~\ref{cor:SottografiQuozientati2}.}\label{Fig:poorer-top}
\end{figure}
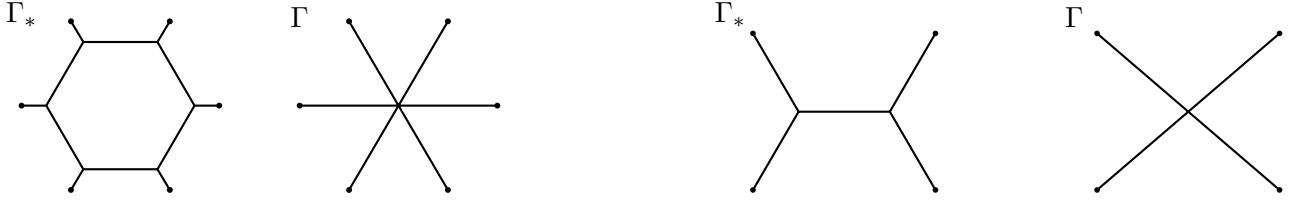

\begin{rem}
By approximation, the results in \cref{cor:ContieneCopiaG}, \cref{cor:SottografiQuozientati1}, and \cref{cor:SottografiQuozientati2} hold also in case the maps $\Gamma|_{E_i}$ of the comparison networks are just Lipschitz immersions.
\end{rem}

\section{Local minimality among partitions}\label{Sec:partition}

Let $\Omega\subset \mathbb{R}^2$ be open. We denote by $P(E,\Omega):=\vert D\chi_{E}\vert(\Omega)$ the (relative) perimeter of a measurable set $E\subset \Omega$ in $\Omega$.  The symbol $\partial^*E$ denotes the reduced boundary of $E$, namely the set of points $x\in {\rm spt} |D\chi_E|$ where the generalized outer unit normal to $E$ exists. Recall that $|D\chi_E|$ is concentrated on $\partial^* E$. For the theory of sets of finite perimeter and functions of bounded variation we refer the reader to \cite{afp}.

We say that  $\mathbf{E}=(E_1,\ldots,E_n)$, for $n \in \N$ with $n\ge 2$,
is Caccioppoli partition of $\Omega$ if $E_i \subset \Omega$ for any $i$, $|E_i\cap E_j|=0$ for $i\neq j$, $|\Omega \setminus \cup_{i=1}^n E_i |=0$, and $\sum_{i=1}^n P(E_i,\Omega)<+\infty$. We denote by $\Sigma_{ij}:=\partial^*E_i\cap \partial^*E_j$ and by $\nu_{ij}=\nu_i=-\nu_j$
the unit normal to $\Sigma_{ij}$, where $\nu_i$ is the generalized outer unit normal to the set $E_i$. In particular we can think of $\nu_{ij}$ as a normal pointing from $E_i$ into $E_j$.

We shall mostly focus our attention to Caccioppoli partitions $\mathbf{E}=(E_1,E_2,E_3)$ defined by three sets.
We remark that, thanks to~\cite[Theorem 4.17]{afp}, it holds  
$\Ha^1(\partial^*E_1\cap \partial^*E_2 \cap \partial^*E_3)=0$, or equivalently 
$$
\Ha^1(\Sigma_{ij}\cap\Sigma_{ik})=0\;\text{for}\;i,j,k\in \{1,2,3\}\;\text{with}\;j\neq k\,.
$$
Moreover  
$P(E_i,\Omega)
=\Ha^1(\Sigma_{ij}\cup\Sigma_{ik})=\Ha^1(\Sigma_{ij})+\Ha^1(\Sigma_{ik})$, for any distinct $i,j,k$ and then
$$
\frac{1}{2}\sum_{i=1}^3P(E_i,\Omega)
=\Ha^1(\Sigma_{12}\cap \Omega)+\Ha^1(\Sigma_{23}\cap \Omega)+\Ha^1(\Sigma_{31}\cap \Omega)\,.
$$

We recall from \cite[Theorem 3.87]{afp} that $BV$ functions admit \emph{(inner) traces} on boundaries of open sets with Lipschitz boundary. More precisely, if $\Omega\subset \R^2$ is a bounded open set with Lipschitz boundary and $u \in [BV(\Omega)]^n$, then for $\Ha^1$-a.e. point $x \in \partial \Omega$ there exists $\tr_\Omega u(x) \in \R^n$ such that
\[
\lim_{r\to0^+} \frac{1}{r^2} \int_{B_r(x) \cap \Omega} |u(y) - \tr_\Omega u(x)| \de y =0.
\]
Moreover the function $\tr_\Omega u$ belongs to $[L^1(\partial\Omega,\Ha^1)]^n$.

If $F\subset \R^2$ is measurable and $u:F\to \R^n$ is a (possibly vector-valued) measurable function, we say that $u$ has \emph{approximate limit} $\text{ap -}\lim_{y\to x} u = u_0 \in \R^n$ \emph{at $x \in F$} if
\[
\lim_{r\to0^+} \frac{1}{r^2} \int_{B_r(x) \cap F} |u(y) - u_0| \de y =0.
\]
Adopting the above terminology, $\tr_\Omega u (x)= \text{ap -}\lim_{y\to x}u$ for $\Ha^1$-a.e. $x \in \partial \Omega$, for $u \in BV(\Omega)$ with $\Omega$ Lipschitz and bounded.

\medskip

Let us now fix a bounded open set $\Omega\subset\mathbb{R}^2$ with Lipschitz boundary
and let $\widehat{\mathbf{E}}=(\widehat{E}_1,\widehat{E}_2,\widehat{E}_3)$ be a fixed Caccioppoli partition of $\Omega$.
We define
\begin{displaymath}
\mathcal{A} := \left\{\mathbf{E}=(E_1,E_2,E_3) \st \mathbf{E} \mbox{ is a Caccioppoli partition of }\Omega\,,\, 
\tr_\Omega \chi_{E_i} = \tr_\Omega \chi_{\widehat{E}_i} \,\,\Ha^1\text{-a.e. on $\partial\Omega$}
\right\}\,.
\end{displaymath}
Clearly $\widehat{\mathbf{E}}$ itself is an element of $\mathcal{A}$.
For every $\mathbf{E}\in\mathcal{A}$ we introduce the energy
\begin{equation*}
\mathcal{P}(\mathbf{E}):= \frac{1}{2}\sum_{i=1}^3 P(E_i,\Omega)\,.
\end{equation*}
A partition $\widetilde{\mathbf{E}}\in \mathcal{A}$ is a \emph{minimizer of $\mathcal{P}$} in $\Omega$ if
\begin{equation*}
\mathcal{P}( \widetilde{\mathbf{E}}) \leq \mathcal{P}(\mathbf{E})\,.
\end{equation*}
for every $\mathbf{E}\in \mathcal{A}$.

\begin{defn}[Approximately regular vector fields on $\R^2$]
Let $\Omega\subset \R^2$ be a bounded open set with Lipschitz boundary. A measurable vector field $\Phi:\overline{\Omega}\to \R^2$ is said to be \emph{approximately regular} if it is bounded and for every Lipschitz curve\footnote{Here by Lipschitz curve we mean the image of a Lipschitz embedding $\sigma:[0,1]\to \overline{\Omega}$.} $\gamma \subset \overline{\Omega}$ there holds
\[
\text{ap -}\lim_{y\to x} \big( \Phi(y)\cdot \nu_\gamma(x) \big) = \Phi(x)\cdot \nu_\gamma(x),
\]
at $\Ha^1$-a.e. $x \in \gamma$, for any chosen unit normal to $\gamma$ at $x$.
\end{defn}
We recall that we are assuming measurable functions to be defined pointwise, hence $\Phi(x)\cdot \nu_\gamma(x)$ is a well defined number in the above definition.

\medskip

In this paper, when we talk about the divergence of a bounded vector field 
$\Phi:\Omega\to\R^2$ we refer to the distributional divergence of $\Phi$, i.e. $\int_\Omega u\, \div \Phi \eqdef - \int_\Omega \nabla u\cdot \Phi \de x$ for any $u \in C^1_c(\Omega)$. We say that $\div\Phi \in L^p(\Omega)$ whenever the distributional divergence is represented by an $L^p(\Omega)$-function, still denoted by $\div \Phi$.\\
If $u \in BV(\Omega)$ we denote by $Du=(\partial_1u,\partial_2u)$ the vector valued measure satisfying $\sum_i\int_\Omega \Phi_i\de (\partial_i u) = - \int_\Omega u\, \div\Phi \de x$ for any $\Phi\in C^1_c(\Omega)$, and we denote $\int_\Omega \Phi\cdot Du \eqdef \sum_i\int_\Omega \Phi_i\de (\partial_i u)$ for any field $\Phi$ such that the latter expression makes sense.

\begin{lemma}[{Divergence theorem for approximately regular vector field, \cite[Lemma 2.4]{AlBuDa}}]\label{divergence-theorem}
Let $\Omega$ be a bounded open set in $\mathbb{R}^2$ with Lipschitz boundary and let $\nu_{\partial\Omega}$ be its inner normal, which is well-defined $\Ha^1$-ae on $\partial\Omega$. Let $\Phi:\overline{\Omega}\to \R^2$
be an approximately regular vector field and let $u\in BV(\Omega)$. Assume that 
$\div \Phi\in L^\infty(\Omega)$ and $\tr_\Omega u \,\Phi\in L^1(\partial\Omega,\Ha^1)$. Then
\begin{equation}\label{thm-div}
    \int_{\Omega} \Phi\cdot Du=-\int_{\Omega} u\, \div\Phi \de x
    -\int_{\partial\Omega} \tr_\Omega u\,\,\,\Phi\cdot\nu_{\partial\Omega} \de\Ha^1\,,
\end{equation}
where $\Phi\cdot\nu_{\partial\Omega} (x) = \text{\rm ap -}\lim_{y\to x} ( \Phi(y)\cdot \nu_\gamma(x) ) $ at $\Ha^1$-a.e. $x \in \partial\Omega$.
\end{lemma}

In what follows we will apply \cref{divergence-theorem} to characteristic functions
of sets of finite perimeter.

\begin{cor}\label{null-lagrangian}
Let $\Omega\subset\mathbb{R}^2$ be a bounded open set with Lipschitz 
boundary
and let $\mathbf{E},\widetilde{\mathbf{E}}\in \mathcal{A}$. Let $\Phi: \overline{\Omega}\to\mathbb{R}^2$ be an approximately regular vector field with $\div \Phi = 0$.
Then 
\begin{equation*}
 \int_{\Omega} \Phi\cdot D\chi_{\widetilde{E}_i}=
 \int_{\Omega} \Phi\cdot D\chi_{E_i}\,,
\end{equation*}
for any $i=1,2,3$.
\end{cor}
\begin{proof}
Using the fact that $\div \Phi = 0$, formula~\eqref{thm-div} applied with $u=\chi_{\widetilde{E}_i}$ or $u= \chi_{E_i}$ reduces to
\begin{align*}
\int_{\Omega} \Phi\cdot D\chi_{\widetilde{E}_i}=-\int_{\partial\Omega} \tr_\Omega\chi_{\widetilde{E}_i}\,\,\Phi\cdot\nu_{\partial\Omega} \de\Ha^1\,,\\
\int_{\Omega} \Phi\cdot D\chi_{E_i}=-\int_{\partial\Omega} \tr_\Omega\chi_{E_i}\,\,\Phi\cdot\nu_{\partial\Omega} \de\Ha^1\,,
\end{align*}
for any $i=1,2,3$. By definition of the set $\mathcal{A}$, the two right hand sides above are equal, and the claim follows.
\end{proof}

Roughly speaking, the next lemma states the well-known fact that if a piecewise smooth vector field is such that its normal component along its jump set is well-defined, then the distributional divergence is simply given by the pointwise divergence of the field computed where it is smooth. We provide a proof for the convenience of the reader.

\begin{lemma}\label{lem:DivergenceNormalComponents}
Let $\Omega$ be a bounded open set with Lipschitz boundary. Let $\Psi:\overline\Omega\to\R^2$ be a measurable bounded vector field. Suppose that that there exists a Caccioppoli partition $\mathbf{F}=(F_1,\ldots,F_n)$ of $\Omega$, where $F_i$ is open with Lipschitz boundary, such that $\Psi|_{F_i} \in C^1(F_i)$ with $\div(\Psi|_{F_i}) \in L^\infty(F_i)$ and
\begin{equation}\label{condizione-tracce}
    \tr_{F_i}(\Psi|_{F_i}) \cdot \nu_i = - \tr_{F_j}(\Psi|_{F_j}) \cdot \nu_j,
\end{equation}
for any $i\neq j$, at $\Ha^1$-a.e. point on $\partial^*F_i\cap \partial^* F_j$, where $\nu_i$ is the $\Ha^1$-a.e. defined outer normal to $F_i$.

Then the distributional divergence of $\Psi$ is given by
\begin{equation*}
    \div\Psi = \sum_{i=1}^n \chi_{F_i}\div(\Psi|_{F_i}) \de x.
\end{equation*}
If also $\div(\Psi|_{F_i})=0$ on $F_i$ for any $i$, then $\div \Psi =0$.
\end{lemma}

\begin{proof}
By assumptions, we see that $\Psi|_{F_i} \in [BV(F_i)]^2$. Since $BV$ functions admit inner traces on boundaries of Lipschitz domains, we see that the field $\Psi_i:\overline{F_i}\to \R^2$ defined by
\[
\Psi_i(x) \eqdef \begin{cases}
\Psi(x) & x \in F_i,\\
\tr_{F_i} (\Psi|_{F_i}) & \Ha^1\text{-a.e. on $\partial^*F_i$,}
\end{cases}
\]
and arbitrarily defined on $\partial F_i$ where $\tr_{F_i} \Psi|_{F_i}$ does not exist, is approximately regular.

Let $u \in C^1_c(\Omega)$. Applying \cref{divergence-theorem} on each $F_i$, we can compute
\[
\begin{split}
    \int_\Omega \Psi\cdot\nabla u \de x 
    &= \sum_i \int_{F_i} \nabla u \cdot \Psi_i \de x 
    =  -\sum_i \int_{F_i} u \, \div (\Psi|_{F_i}) \de x + \sum_i \int_{\partial^*F_i} u \,\Psi_i \cdot \nu_i \de \Ha^1\\
    &= -\sum_i \int_{F_i} u \, \div (\Psi|_{F_i}) \de x .
\end{split}
\]
\end{proof}

\begin{defn}[Local paired calibration]\label{paired}
A \emph{local paired calibration} for a Caccioppoli partition $\mathbf{E}=(E_1,E_2,E_3)$ is a collection of three approximately regular 
vector fields $\Phi_1,\Phi_2,\Phi_3: \overline{\Omega} \rightarrow \R^2$ such that
\begin{enumerate}[label={\normalfont{\color{blue}\arabic*)}}]
\item \label{it:r1} $\div \Phi_i = 0$ \quad for $i=1,2,3$,
\item \label{it:r2}$\vert \Phi_i - \Phi_j\vert \leq 1$ \quad $\Ha^1$-a.e. in $\Omega$, for $i,j =1,2,3$, $i\neq j$,
\item \label{it:r3}$(\Phi_i - \Phi_j)\cdot \nu_{ij} = 1$ \quad $\Ha^1$-a.e. in  $\Sigma_{ij}$, for $i,j =1,2,3$, $i\neq j$.
\end{enumerate}  
\end{defn} 

\begin{rem}
The concept of local paired calibration is nothing but the notion
of paired calibration introduced by Lawlor and Morgan in~\cite{lawmor}
restricted to the case of a partition of $\Omega$ composed of three sets.
\end{rem}

\begin{rem}\label{differenze}
Let $\Omega$ be bounded, let $\mathbf{E} \in \mathcal{A}$ and assume that there exist approximately regular fields $\Psi_{12},\Psi_{23},\Psi_{31}:\overline\Omega\to \R^2$ such that
\begin{itemize}
    \item $\div \Psi_{12}=\div\Psi_{23}=\div\Psi_{31}=0$,
    \item $|\Psi_{12}|,|\Psi_{23}|,|\Psi_{31}|\le 1$  $\Ha^1$-a.e. in $\Omega$,
    \item $\Psi_{ij}\cdot \nu_{ij}=1$ $\Ha^1$-a.e. in $\Sigma_{ij}$, for $i,j=1,2,3$ such that $\Psi_{ij}$ is defined,
    \item $\Psi_{12}+\Psi_{23}+\Psi_{31}=0$ $\Ha^1$-a.e. in $\Omega$.
\end{itemize}
Then there exists a local paired calibration for $\mathbf{E}$.

Indeed, fix an arbitrary approximately regular field $\Phi_1$ with $\div \Phi_1=0$, for example, $\Phi_1(x,y)\eqdef (0,0)|_{\overline\Omega}$. Hence set $\Phi_2\eqdef \Phi_1-\Psi_{12}$ and $\Phi_3\eqdef \Psi_{31}+\Phi_1$. Then $\Phi_i-\Phi_j = \Psi_{ij}$ $\Ha^1$-a.e. in $\Omega$.
Hence the properties of $\Psi_{12},\Psi_{23},\Psi_{31}$ imply that $\Phi_1,\Phi_2,\Phi_3$ give a local paired calibration for $\mathbf{E}$.
\end{rem}

The next lemma states that existence of a local paired calibration for a partition implies that it is minimizing. We provide a detailed proof which formalizes the general principle introduced in \cite[Section 1.1, Section 1.2]{lawmor}.

\begin{lemma}
If $\Phi$ is a local paired calibration for $\widetilde{\mathbf{E}}\in \mathcal{A}$,
then $\widetilde{\mathbf{E}} \in \mathcal{A}$ is a minimizer of $\mathcal{P}$ in $\Omega$.
\end{lemma}

\begin{proof}
Using the third condition in the definition of local paired calibration, denoting by $\widetilde{\nu}_{ij}$ the usual normal at the interfaces of $\widetilde{\mathbf{E}}$, and
the structure of Caccioppoli partitions~\cite[Theorem 4.17]{afp} we have
\begin{align*}
\mathcal{P}(\widetilde{\mathbf{E}})&=\frac12\sum_{i=1}^3 P(\widetilde{E}_i,\Omega)
=\Ha^1(\widetilde{\Sigma}_{12}\cap \Omega)+\Ha^1(\widetilde{\Sigma}_{23}\cap \Omega)+\Ha^1(\widetilde{\Sigma}_{31}\cap \Omega)\\
& \overset{\ref{it:r3}}{=}\int_{\widetilde{\Sigma}_{12}\cap\Omega}(\Phi_{1}-\Phi_{2})\cdot\widetilde\nu_{12}\,\,\mathrm{d}\Ha^1
+\int_{\widetilde{\Sigma}_{23}\cap\Omega}(\Phi_{2}-\Phi_{3})\cdot\widetilde\nu_{23}\,\,\mathrm{d}\Ha^1
+\int_{\widetilde{\Sigma}_{31}\cap\Omega}(\Phi_{3}-\Phi_{1})\cdot\widetilde\nu_{31}\,\,\mathrm{d}\Ha^1\\
&=  \sum_{i=1}^3 \int_{\Omega} \Phi_i\cdot D\chi_{\widetilde{E}_i}\,.
\end{align*}
Thanks to the divergence free condition on the vector fields we can apply 
\cref{null-lagrangian} getting
\begin{align*}
 \sum_{i=1}^3 \int_{\Omega} \Phi_i\cdot D\chi_{\widetilde{E}_i}
    \overset{\ref{it:r1}}{=} \sum_{i=1}^3 \int_{\Omega} \Phi_i\cdot D\chi_{E_i}\,.
\end{align*}
To conclude we use the second condition in the definition of local paired calibration to get
\begin{align*}
\sum_{i=1}^3 \int_{\Omega} &\Phi_i\cdot D\chi_{E_i}
\\
&=\int_{{\Sigma}_{12}\cap\Omega}(\Phi_{1}-\Phi_{2})\cdot\nu_{12}\,\,\mathrm{d}\Ha^1
+\int_{{\Sigma}_{23}\cap\Omega}(\Phi_{2}-\Phi_{3})\cdot\nu_{23}\,\,\mathrm{d}\Ha^1
+\int_{{\Sigma}_{31}\cap\Omega}(\Phi_{3}-\Phi_{1})\cdot\nu_{31}\,\,\mathrm{d}\Ha^1\\
&\leq \int_{{\Sigma}_{12}\cap\Omega}\vert\Phi_{1}-\Phi_{2}\vert\,\,\mathrm{d}\Ha^1
+\int_{{\Sigma}_{23}\cap\Omega}\vert\Phi_{2}-\Phi_{3}\vert\,\,\mathrm{d}\Ha^1
+\int_{{\Sigma}_{31}\cap\Omega}\vert\Phi_{3}-\Phi_{1}\vert\,\,\mathrm{d}\Ha^1\\
&\overset{\ref{it:r2}}{\leq} \Ha^1(\Sigma_{12}\cap \Omega)+\Ha^1(\Sigma_{23}\cap \Omega)+\Ha^1(\Sigma_{31}\cap \Omega)
=\frac{1}{2}\sum_{i=1}^3 P(E_i,\Omega)=\mathcal{P}(\mathbf{E})\,.
\end{align*}
Following the chain of inequalities we have $\mathcal{P}(\widetilde{\mathbf{E}})\leq \mathcal{P}(\mathbf{E})$ as desired.
\end{proof}

We conclude by proving the \emph{local} minimality of a partition induced by a minimal network.

\begin{teo}\label{esibizione-calibrazione}
Let $\Gamma_*:G\to \R^2$ be a minimal network such that $\Gamma_*(G)\subset D$ where $D$ is a domain of class $C^1$ homeomorphic to a closed disk, and $\Gamma_*(G) \cap \partial D$ is the set of endpoints of $\Gamma_*$.

Then there exists a bounded open set $\Omega'\subset \R^2$ such that $\Gamma_*(G)\subset \Omega'$, $\Omega\eqdef \Omega'\cap {\rm int\,}(D)$ has Lipschitz boundary, there exists a Caccioppoli partition $\widetilde{\mathbf{E}}=(\widetilde{E}_1,\widetilde{E}_2,\widetilde{E}_3)$ of $\Omega$ such that $\Omega \cap \cup_i\partial\widetilde{E}_i =\Omega \cap \Gamma_*(G)$ and there exists a local paired calibration for $\widetilde{\mathbf{E}}$ in $\Omega$.

In particular, $\widetilde{\mathbf{E}}$ is a minimizer for $\mathcal{P}$ in $\mathcal{A}$, that is the class of partitions having the same trace of $\widetilde{\mathbf{E}}$ on $\partial\Omega$.
\end{teo}

\begin{proof}
Let us assume that $G$ has at least two junctions. The cases with no junctions or with one junction are easier to treat with the construction outlined in the rest of the proof.

We first consider a simple configuration
in which $\Gamma_*(G)$
is composed of five curves that meet at two triple junctions. 
Let $\Gamma'_*(G)$ be the network having all edges equal to those of $\Gamma_*(G)$ except for those ending at an endpoint, 
which are 
lengthened by an additive length equal to $\delta'\in(0,1)$ (that will be suitably chosen later).
We first construct a calibration for the Caccioppoli partition identified by the new network $\Gamma'_*(G)$ containing $\Gamma_*(G)$, over an open set $\Omega'$ containing $\Gamma_*(G)$ (see \cref{fig:simple-minimal}). Eventually, in the general case, a restriction of both the partition and the calibration to $D$ will give the desired calibration of a partition of the final set $\Omega$. 
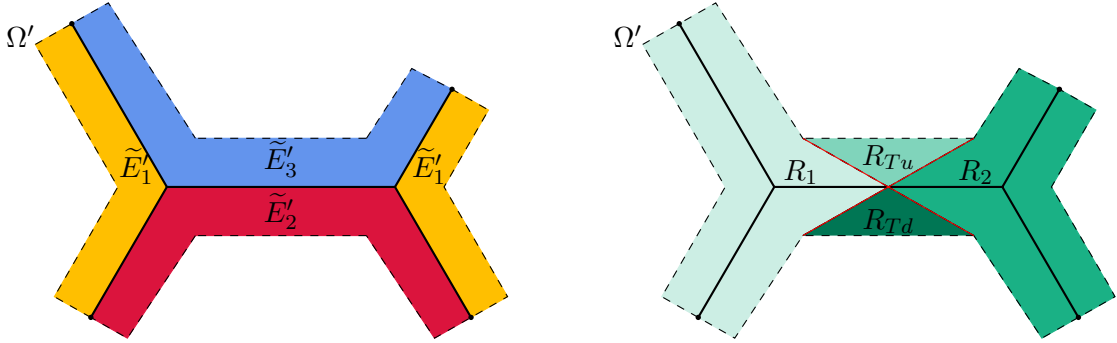
\begin{figure}[H]
    \centering
    \begin{tikzpicture}[scale=0.5]
\filldraw[amber]
(8,-3.46)
--(8.96,-2.91)--(7.29,0)--
 (7.29,0)--(8.46,2.04)--
(7.5,2.59) --(6,0);    
\filldraw[amber]
(-2,-3.46)--(-2.96,-2.91)--
(-1.29,0)--
(-3.46,3.78)--(-2.5,4.33)--(0,0);
\filldraw[crimson]
(-2,-3.46)--(-1.04,-4.01)--(0.75,-1.29)--(5.25,-1.29)--
(7.04,-4.01)--(8,-3.46)--(6,0)--(0,0)--(-2,-3.46);
\filldraw[cornflowerblue]
(5.25,1.29)--(0.75,1.29)--
(-1.54,4.88)-- (-2.5,4.33)--
(-2.5,4.33)--(0,0)--(6,0)--(7.5,2.59)--(6.44,3.14);
\draw[dashed]
(-2,-3.46)--(-2.96,-2.91)--
(-1.29,0)--(-3.46,3.78);
\draw[dashed]
(-3.46,3.78)--(-2.5,4.33)--
(-2.5,4.33)--(-1.54,4.88)--(0.75,1.29);
\draw[dashed]
(0.75,-1.29)--(-1.04,-4.01)--
(-2,-3.46);
\draw[dashed]
(5.25,1.29)--
 (6.44,3.14)--(7.5,2.59)--
(8.46,2.04)--
 (7.29,0)--(8.96,-2.91)--(8,-3.46)--
(8,-3.46)--(7.04,-4.01)--(5.25,-1.29);
\draw[thick]
(0,0)--(-2.5,4.33)
(0,0)--(-2,-3.46)
(0,0)--(6,0)
(6,0)--(7.5,2.59)
(6,0)--(8,-3.46);
\draw[dashed]
(5.25,1.29)--(0.75,1.29)
(5.25,-1.29)--(0.75,-1.29);
\fill[black](-2.5,4.33) circle (2pt);  
\fill[black](-2,-3.46) circle (2pt);  
\fill[black](7.5,2.59) circle (2pt);  
\fill[black](8,-3.46) circle (2pt); 
\path
(-4.5,4) node[right]{$\Omega'$}
(-0.8,-0.2) node[above]{$\widetilde{E}_1'$}
(6.9,-0.2) node[above]{$\widetilde{E}_1'$}
(3,-1.4) node[above]{$\widetilde{E}_2'$}
(3,0) node[above]{$\widetilde{E}_3'$};
\end{tikzpicture} \qquad\quad
\begin{tikzpicture}[scale=0.5]
\filldraw[green(munsell)!20!white]
(-3.46,3.78)--(-2.5,4.33)--
(-2.5,4.33)--(-1.54,4.88)--
(0.75,1.29)--(3,0)--(0.75,-1.29)--(-1.04,-4.01)--
(-2,-3.46)--(-2.96,-2.91)--
(-1.29,0)--(-3.46,3.78);
\filldraw[green(munsell)!50!white]
(5.25,1.29)--(3,0)--(0.75,1.29);
\filldraw[green(munsell)!90!white]
(3,0)--(5.25,1.29)--(6.44,3.14)--(7.5,2.59)--
(8.46,2.04)--
 (7.29,0)--
 (8.96,-2.91)--(8,-3.46)--
(7.04,-4.01)--
 (5.25,-1.29)--(3,0);
\filldraw[green(munsell)!70!black]
(5.25,-1.29)--(3,0)--(0.75,-1.29);
\draw[thick]
(0,0)--(-2.5,4.33)
(0,0)--(-2,-3.46)
(0,0)--(6,0)
(6,0)--(7.5,2.59)
(6,0)--(8,-3.46);
\draw[dashed]
(5.25,1.29)--(0.75,1.29)
(5.25,-1.29)--(0.75,-1.29);
\draw[dashed]
(-3.46,3.78)--(-2.5,4.33)--
(-2.5,4.33)--(-1.54,4.88)--
(0.75,1.29)--(3,0)--(0.75,-1.29)--(-1.04,-4.01)--
(-2,-3.46)--
(-2,-3.46)--(-2.96,-2.91)--
(-1.29,0)--(-3.46,3.78);
\draw[dashed]
(5.25,1.29)--(3,0)--(0.75,1.29);
\draw[dashed]
(3,0)--(5.25,1.29)--(6.44,3.14)--
 (7.5,2.59)--(8.46,2.04)--
 (7.29,0)--(8.79,-2.59)--
 (8.96,-2.91)--(8,-3.46)--
(7.04,-4.01)--
 (5.25,-1.29)--(3,0);
\draw[dashed]
(5.25,-1.29)--(3,0)--(0.75,-1.29);
\draw[red]
(3,0)--(0.75,-1.29)
(3,0)--(0.75,1.29)
(3,0)--(5.25,-1.29)
(3,0)--(5.25,1.29);
\fill[black](-2.5,4.33) circle (2pt);  
\fill[black](-2,-3.46) circle (2pt);  
\fill[black](7.5,2.59) circle (2pt);  
\fill[black](8,-3.46) circle (2pt); 
\path
(-4.5,4) node[right]{$\Omega'$}
(0.7,-0.2) node[above]{$R_1$}
(5.3,-0.2) node[above]{$R_2$}
(3,-1.5) node[above]{$R_{Td}$}
(3,0.2) node[above]{$R_{Tu}$};
\end{tikzpicture}    
\caption{Left: Partition $\widetilde{\mathbf{E}}'=(\widetilde{E}_1',\widetilde{E}_2',\widetilde{E}_3')$ of $\Omega'$. Right: Split of $\Omega'$ in four regions $R_1,R_2,R_{Tu},R_{Td}$.}
    \label{fig:simple-minimal}
\end{figure}
We call $O^1,O^2$ the two junctions and we let
$d$ be the distance between $O^1$ and $O^2$.
We define the set $\Omega'$ to be the $\delta$-tubular neighborhood of $\Gamma'_*(G)$ with $\delta=\frac{d\sqrt{3}}{8}$
truncated at each of the four  new endpoints with the line passing through the endpoint that is orthogonal to the corresponding edge.
We thus define the partition $\widetilde{E}_1',\widetilde{E}_2',\widetilde{E}_3'$ as in \cref{fig:simple-minimal} on the left.

Let $M$ be the midpoint between $O^1$ and $O^2$.
We consider two lines $\ell_1,\ell_2$ through $M$ forming an angle of $30$ degrees with the segment connecting $O^1$ and $O^2$.
By the choice of $\delta$, the two lines meet $\partial\Omega'$ at points whose projections on the edge $O^1\,O^2$ lie in the interior of the edge; hence the lines split $\Omega'$ in four open regions that we denote $R_1$, $R_2$, $R_{Tu}$, $R_{Td}$, as in \cref{fig:simple-minimal}.
Now we construct a calibration for $\widetilde{\mathbf{E}}'$.
Thanks to \cref{differenze} we can 
exhibit directly the ``difference fields''
$\Psi_{12}$, $\Psi_{23}$, $\Psi_{31}$. We define the fields pointwise on each region by
\begin{equation*}
\begin{array}{lllll}
&\textbf{Region $R_1$}:
&\Psi_{12} = \left(\sqrt{3}/2, -1/2\right) , & \Psi_{23} = (0, 1),& \Psi_{31} = \left(-\sqrt{3}/2, -1/2\right) \,,\\
&\textbf{Region $R_2$}:
&\Psi_{12}= \left(-\sqrt{3}/2, -1/2\right) ,& \Psi_{23} = (0, 1),& \Psi_{31} = \left(\sqrt{3}/2, -1/2\right) \,,\\
&\textbf{Region $R_{T_u}$}:
&\Psi_{12} = \left(0,0\right) , & \Psi_{23} = (0, 1),& \Psi_{31} = \left(0,-1\right) \,,\\
&\textbf{Region $R_{T_d}$}:
&\Psi_{12}= \left(0,-1\right) , & \Psi_{23} = (0, 1),& \Psi_{31} = \left(0,0\right) \,.
\end{array}
\end{equation*}
Moreover the fields $\Psi_{ij}$ are pointwise defined on $\partial\Omega' \cup \ell_1\cup\ell_2$ by taking traces from suitable corresponding regions.
More precisely, we set $\Psi_{ij}|_{\partial R \setminus\{M\}} =\tr_R (\Psi_{ij}|_R)$ for regions $R \in \{ R_1, R_2\}$, and  $\Psi_{ij}|_{\partial\Omega' \cap \overline{S} \setminus (\overline{R_1} \cup \overline{R_2})} = \tr_S (\Psi_{ij}|_S)$ for regions $S \in \{ R_{Tu}, R_{T_d}\}$, and $\Psi_{ij}(M)=0$.
\begin{figure}[H]
    \centering
\input{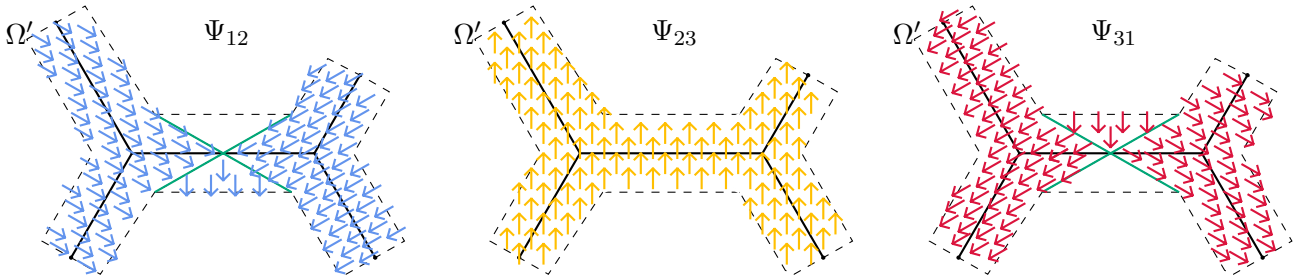}
    \caption{Picture of the three fields $\Psi_{12},\Psi_{23},\Psi_{31}$.}
    \label{fig:calibrazione}
\end{figure}
To prove that $\Psi_{12},\Psi_{23},\Psi_{31}$ induce
a local paired calibration, we check the four conditions
of \cref{differenze}.

The divergence free condition immediately follows from \cref{lem:DivergenceNormalComponents}. Indeed, the fields
are constant in the regions $R_1$, $R_2$, $R_{T_u}$ and $R_{T_d}$, so within each regions they have zero divergence.
Moreover one can easily prove that the condition~\eqref{condizione-tracce} is satisfied.
As an example we check it for $\Psi_{12}$ in the transition
between $R_1$ and $R_{T_u}$.

Since $\nu_{R_1}=(1/2,\sqrt{3}/2)$ and $\nu_{R_{T_u}}=-\nu_{R_1}=(-1/2,-\sqrt{3}/2)$, by direct computation we have:
\begin{align*}
   \tr_{R_1}(\Psi_{12}|_{R_1}) \cdot \nu_{R_1} =(\sqrt{3}/2,-1/2)\cdot (1/2,\sqrt{3}/2)=0\,,\\
   \tr_{R_{T_u}}(\Psi_{12}|_{R_{T_u}}) \cdot \nu_{R_{T_u}}
   =(0,0)\cdot (-1/2,-\sqrt{3}/2)=0
\end{align*}
hence $\tr_{R_1}(\Psi_{12}|_{R_1}) \cdot \nu_{R_1}=-\tr_{R_{T_u}}(\Psi_{12}|_{R_{T_u}}) \cdot \nu_{R_{T_u}}$.
The computations are completely analogous 
for all the other transitions and fields.\\
Moreover, we have $\vert \Psi_1\vert=\vert\Psi_2\vert=\vert\Psi_3\vert=1$.
Finally, the third and fourth condition of  \cref{differenze} are trivial to check.

\medskip

We now pass to the general case.  Let $\Gamma_*, D$ be as in the statement. Denote by $\{D_j\}$ the collection of the connected components of ${\rm int} \left( D\right)\setminus \Gamma_*(G)$. For any $j$, at most six edges of $\Gamma_*(G)$ are contained in the boundary $\partial D_j$, and in particular $\partial D_j$ is a hexagon whenever $\partial D_j \cap \partial D=\emptyset$. Indeed, if at least seven edges of $\Gamma_*(G)$ are contained in the boundary $\partial D_j$, since $D_j$ and $\partial D_j$ are connected, we can split $\partial D_j$ as a union $\partial D_j= \partial (\R^2 \setminus D_j) \cup \bigsqcup_k S_k^i$, where each $S^i_k$ is (the image of) a minimal network with only one endpoint, such endpoint being the intersection of $S^i_k$ with $\partial (\R^2 \setminus D_j)$. But minimal networks $S$ with only one endpoint do not exist, otherwise the distance function from the endpoint would not achieve a maximum on $S$.
Nonexistence of minimal networks with one endpoint also implies that if a path in ${\rm int}\left( D\right)$ crosses an edge of $\Gamma_*(G)$, then it passes from a component $D_j$ to a different component $D_{j'}$. Similarly, one notices that minimal networks with only two endpoints (and at least a triple junction) do not exist, and thus either $\partial D_i \cap \partial D_k=\emptyset$ or $\partial D_i \cap \partial D_k$ is an edge of $\Gamma_*(G)$, for $i\neq k$.

We define a planar graph $Q$ whose vertices are the components $\{D_j\}$ and we say that $D_j$ and $D_k$, for $j\neq k$, are connected by an edge if and only if $\partial D_i \cap \partial D_k$ is an edge of $\Gamma_*(G)$. By the above observations it readily follows that $Q$ is homeomorphic to a subset of the planar triangular graph $Q'$ identified by the lattice $\{ n(1,0) + m(1/2, \sqrt{3}/2)\st n,m \in \Z\}$. In fact, a homeomorphism can be constructed by iteration as follows. Starting from a first component $D_{j_1}$, if $x \in \partial D_{j_1}$ is a triple junction of $\Gamma_*(G)$, parametrizing a circle of small radius centered at $x$ starting from a point in $D_{j_1}$ identifies a triangle in $Q$. Repeating the construction for the other triple junctions lying on $\partial D_{j_1}$, which are six at most, and then iterating the construction on adjacent components eventually leads to the desired homeomorphism.

Since the triangular graph $Q'$ can be colored with three colors, there exists a Caccioppoli partition $\mathbf{F}=(F_1,F_2,F_3)$, consisting of just three sets, of $D$ such that ${\rm int\,}(D) \cap \Gamma_*(G)={\rm int\,}(D) \cap \cup_i \partial F_i$.

Now let $d$ be the minimal length among the edges of $\Gamma_*$. For $\delta\in(0,\frac{d\sqrt{3}}{8})$ small enough and $\delta'\in(0,1)$ that will be suitably chosen we consider the new network $\Gamma'_*(G)$ as before having the same edges of $\Gamma_*(G)$ except for those connected to endpoints which are extended by a length equal to $\delta'$, and we take $\Omega'$ equal to the $\delta$-tubular neighborhood of $\Gamma_*'(G)$, orthogonally truncated at endpoints. By assumptions, since $D$ has boundary of class $C^1$, for almost every $\delta',\delta$ suitably small, $\partial \Omega'$ intersects $\partial D$ transversely. Therefore, choosing any such  $\delta',\delta$, the set $\Omega\eqdef \Omega' \cap {\rm int\,}(D)$ is an open set with Lipschitz boundary. We define the desired Caccioppoli partition $\widetilde{\mathbf{E}}=(\widetilde E_1,\widetilde E_2,\widetilde E_3)$ of $\Omega$ by setting $\widetilde E_i\eqdef F_i \cap \Omega'$ for any $i=1,2,3$.

It remains to exhibit a local paired calibration for $\widetilde{\mathbf{E}}$. Along any edge that is not connected to an endpoint, we locally perform the same construction of \cref{fig:simple-minimal} for the network $\Gamma'_*$. More precisely, at any midpoint $M_{ij}$ between 
any two triple junctions 
$O^i$ and $O^j$ connected by an edge $L$, we consider two lines through $M_{ij}$ forming an angle of $30$ degrees with the segment connecting $O^i$ and $O^j$. By the choice of $\delta$, such lines intersect $\partial\Omega'$ at points whose projections along the direction determined by $L$ lie in the interior of $L$. Hence such lines divide $\Omega'$ into open regions where we can set $\Psi_{12}, \Psi_{23}, \Psi_{31}$ to be the suitable rotation of $\pm\frac{2\pi}{3}$
of the vector $(0,1)$ like done in the simpler case in \cref{fig:simple-minimal}.

The only issue that we need to check is the fact that a choice of these fields is coherent along an arbitrary cycle of the network, see \cref{fig:campi-vettoriali}. Since we proved that a cycle in $\Gamma_*$ is a hexagon, this is easily established by directly prescribing a choice of $\Psi_{12}, \Psi_{23}, \Psi_{31}$, depending on the preassigned $\widetilde{E}_1', \widetilde{E}_2', \widetilde{E}_3'$ as done before, along a cycle as in \cref{fig:campi-vettoriali}, which form a partition of $\Omega'$. Up to relabeling, we have that $\widetilde E'_i \cap {\rm int\,}(D) = \widetilde E_i$ for $i=1,2,3$. Hence taking pointwise restrictions on $D$ of the fields defining the calibration for $\widetilde{\mathbf{E}}'$ on $\Omega'$, we get a calibration for $\widetilde{\mathbf{E}}$ on $\Omega$.

\begin{figure}[H]
    \centering
\input{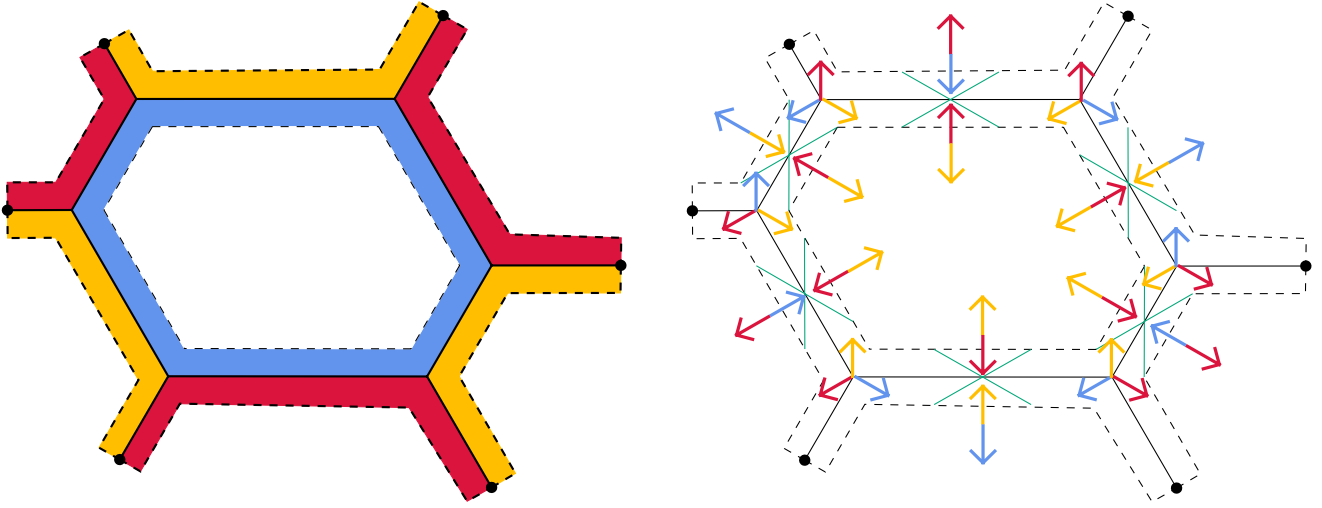}
    \caption{Left: Partition $\widetilde{\mathbf{E}}'=(\widetilde{E}_1', \widetilde{E}_2', \widetilde{E}_3')$, in yellow the set $\widetilde{E}_1'$, in red $\widetilde{E}_2'$ and in blue $\widetilde{E}_3'$. Right: The three fields:
    $\Psi_{12}$ in blue, $\Psi_{23}$ in yellow
    and $\Psi_{31}$ in red.}
    \label{fig:campi-vettoriali}
\end{figure}
\end{proof}

\begin{rem}\label{rem:LocalitaRisultatoPartizioni}
In the construction in \cref{esibizione-calibrazione} of a neighborhood relative to $D$ of a minimal network $\Gamma_*$ together with a Caccioppoli partition minimizing for $\mathcal{P}$, it is \emph{not} possible to choose the neighborhood arbitrarily large. More precisely, recall that in the proof of \cref{esibizione-calibrazione}, $\Omega$ is essentially given by a tubular neighborhood of $\Gamma_*$ of width $\delta<\sqrt{3}d/8$ where $d$ is the length of the shortest edge, intersected with the given domain $D$. Here we construct an example showing that for a choice of a bigger $\delta$, the partition associated to a minimal network $\Gamma_*$ is no longer a minimizer for $\mathcal{P}$ in $\Omega$.

Consider the minimal network $\Gamma_*$ depicted in \cref{fig:controesempio} on the left, together with a $\delta$-tubular neighborhood (orthogonally truncated at endpoints) $\Omega$ and the associated partition
$\mathbf{E}=(E_1,E_2,E_3)$, for $\delta$ to be chosen (comparing to the notation of \cref{esibizione-calibrazione}, here we can take $D$ equal to a suitable $C^1$ topological disk containing $\Omega$).
In this case $\Gamma_*$ is composed of five curves
joining at two triple junction whose distance equals $d$, and we assume that the four edges connected to endpoints have length strictly bigger than $d$.
We consider a comparison partition $\mathbf{F}$ (not induced by a network!) in the class $\mathcal{A}$ of partitions having the same trace of $\mathbf{E}$ on the boundary, as depicted in \cref{fig:controesempio} on the right. The interfaces determining the partition $\mathbf{F}$ are obtained by deleting the central curve of $\Gamma_*$, by shortening the four edges connected to endpoints of the same amount, and then by joining with two horizontal segments the four new endpoints of the shortened edges.
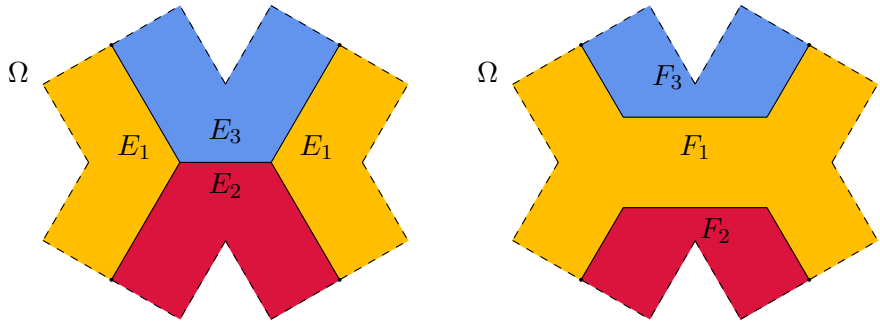
\begin{figure}[H]
    \centering
       \begin{tikzpicture}[scale=0.3]
\filldraw[amber]
(-6,3.45)--(-3,5.19)--(0,0)--(-3,-5.19)--(-6,-3.45)--(-4,0)--(-6,3.45);
\filldraw[amber]
(10,3.45)--(7,5.19)--(4,0)
--(7,-5.19)--(10,-3.45)--
(8,0)--(10,3.45);
\filldraw[cornflowerblue]
(-3,5.19)--(0,0)--(4,0)--(7,5.19)--
(4,6.93)--(2,3.46)--(0,6.93);
\filldraw[crimson]
(-3,-5.19)--(0,0)--(4,0)--(7,-5.19)--
(4,-6.93)--(2,-3.46)--(0,-6.93);
\draw[dashed]
(2,-3.46)--(4,-6.93)
(2,-3.46)--(0,-6.93)
(2,3.46)--(4,6.93)
(2,3.46)--(0,6.93)
(-6,3.45)--(0,6.93)
(-6,-3.45)--(0,-6.93)
(-4,0)--(-6,3.45)
(-4,0)--(-6,-3.45)
(4,6.93)--(10,3.45)
(4,-6.93)--(10,-3.45)
(8,0)--(10,3.45)
(8,0)--(10,-3.45);
\draw
(0,0)--(-3,5.19)
(0,0)--(-3,-5.19)
(0,0)--(4,0)
(4,0)--(7,5.19)
(4,0)--(7,-5.19);
\fill[black](-3,5.19) circle (2.5pt);  
\fill[black](-3,-5.19)circle (2.5pt);  
\fill[black](7,5.19) circle (2.5pt);  
\fill[black](7,-5.19) circle (2.5pt); 
\path
(-8,4) node[right]{$\Omega$}
(-2,-0.3) node[above]{$E_1$}
(6,-0.3) node[above]{$E_1$}
(2,-2) node[above]{$E_2$}
(2,0.4) node[above]{$E_3$};
\end{tikzpicture}\quad\quad
\begin{tikzpicture}[scale=0.3]
\filldraw[amber]
(-6,3.45)--(-3,5.19)--(-1.15,2)--(5.15,2)--
(7,5.19)--(10,3.45)--
(8,0)--(10,-3.45)--(7,-5.19)--(5.15,-2)--
(-1.15,-2)--(-3,-5.19)--(-6,-3.45)--(-4,0)--(-6,3.45);
\filldraw[cornflowerblue]
(-3,5.19)--(-1.15,2)--(5.15,2)--(7,5.19)--
(4,6.93)--(2,3.46)--(0,6.93);
\filldraw[crimson]
(-3,-5.19)--(-1.15,-2)--(5.15,-2)--(7,-5.19)--
(4,-6.93)--(2,-3.46)--(0,-6.93);
\draw[dashed]
(2,-3.46)--(4,-6.93)
(2,-3.46)--(0,-6.93)
(2,3.46)--(4,6.93)
(2,3.46)--(0,6.93)
(-6,3.45)--(0,6.93)
(-6,-3.45)--(0,-6.93)
(-4,0)--(-6,3.45)
(-4,0)--(-6,-3.45)
(4,6.93)--(10,3.45)
(4,-6.93)--(10,-3.45)
(8,0)--(10,3.45)
(8,0)--(10,-3.45);
\draw[black]
(-1.15,2)--(-3,5.19)
(-1.15,-2)--(-3,-5.19)
(-1.15,2)--(5.15,2)
(-1.15,-2)--(5.15,-2)
(5.15,2)--(7,5.19)
(5.15,-2) --(7,-5.19);
\fill[black](-3,5.19) circle (2.5pt);  
\fill[black](-3,-5.19)circle (2.5pt);  
\fill[black](7,5.19) circle (2.5pt);  
\fill[black](7,-5.19) circle (2.5pt); 
\path
(-8,4) node[right]{$\Omega$}
(2,-0.3) node[above]{$F_1$}
(2.9,-4) node[above]{$F_2$}
(0.8,2.8) node[above]{$F_3$};
\end{tikzpicture}
    \caption{Partitions $\mathbf{E}$ and $\mathbf{F}$ in \cref{rem:LocalitaRisultatoPartizioni}.}
    \label{fig:controesempio}
\end{figure}
We claim that if $\delta>\frac{\sqrt{3}}{4}d$, then $\mathbf{F}$ can be chosen so that $\mathcal{P}(\mathbf{F})<\mathcal{P}(\mathbf{E})$, giving an upper bound on the width of the tubular neighborhood where the partition induced by $\Gamma_*$ may be minimizing for $\mathcal{P}$.

Indeed, let $a$ be half of the length of one of the horizontal edges of the interfaces of $\mathbf{F}$ and let $b$ be the length of the deleted portion of one of the four external edges of $\Gamma_*$ (see \cref{fig:ContoControesempio}). By symmetry, if $\tfrac{d}{2}+2b>2a$ then $\mathcal{P}(\mathbf{F})<\mathcal{P}(\mathbf{E})$. Letting $h$ be the distance between the horizontal edge of $\Gamma_*$ and the horizontal part of the interfaces of $\mathbf{F}$, then $b=\tfrac{2}{\sqrt{3}}h$ and $a=\tfrac{d}{2}+\tfrac{b}{2}=\tfrac{d}{2}+\tfrac{h}{\sqrt{3}}$. Hence the condition $\tfrac{d}{2}+2b>2a$ is equivalent to $h>\tfrac{\sqrt{3}}{4}d$. Hence, whenever $\delta>\tfrac{\sqrt{3}}{4}d$, then such a partition $\mathbf{F}$ satisfying $\mathcal{P}(\mathbf{F})<\mathcal{P}(\mathbf{E})$ can be constructed within $\Omega$.
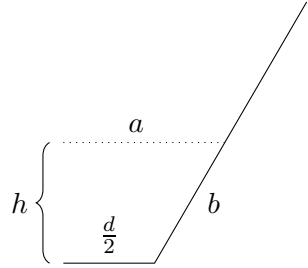
\begin{figure}[H]
\begin{center}
		\begin{tikzpicture}[scale=0.8]
		\draw
		(-1.5,0)--(0,0)to[out=60, in=240, looseness=0.5](2.5,4.33);
		\draw[dotted]
		(-1.5,2)--(1.155,2);
		\path[font=\normalsize]
		(-0.75,0)node[above]{$\frac{d}{2}$};
		\path[font=\normalsize]
		(0.7,1)node[right]{$b$};
		\path[font=\normalsize]
		(-0.3,2)node[above]{$a$};
		\draw [decorate,decoration={brace,amplitude=5pt,raise=1ex}]
        (-1.5,0) -- (-1.5,2)
        node[midway,xshift=-1.5em]{$h$};
		\end{tikzpicture}
	\end{center}
    \caption{Continuous lines: Upper-right quarter of $\Gamma_*$ in \cref{rem:LocalitaRisultatoPartizioni}. Dotted line: half of the upper horizontal edge of the interfaces of $\mathbf{F}$ in \cref{rem:LocalitaRisultatoPartizioni}.}
    \label{fig:ContoControesempio}
\end{figure}
\end{rem}

\bibliographystyle{plain}
\addcontentsline{toc}{section}{References}
\bibliography{references} 
\end{document}